\documentclass[a4paper,12pt]{article}
\usepackage{amssymb,amsmath,amsfonts,amsthm,stmaryrd}

\numberwithin{equation}{section}
\usepackage[left=1 in, right=1 in,top=1 in, bottom=1 in]{geometry}

\providecommand{\abs}[1]{\left\vert#1\right\vert}
\providecommand{\norm}[1]{\left\Vert#1\right\Vert}
\providecommand{\pnorm}[2]{\left\Vert#1\right\Vert_{L^{#2}}}

\providecommand{\Rn}[1]{\mathbb{R}^{#1}}

\providecommand{\jump}[1]{\left\llbracket #1 \right\rrbracket }

\def\wstar{\overset{*}{\rightharpoonup}}
\def\nab{\nabla}
\def\dt{\partial_t}

\def\hal{\frac{1}{2}}

\def\ep{\varepsilon}

\def\rest{\hskip 1pt{\hbox to 10.8pt{\hfill\vrule height 7pt width 0.4pt depth 0pt\hbox{\vrule height 0.4pt
width 7.6pt depth 0pt}\hfill}}}

\def\evalu{\hskip 1pt{\hbox to 2pt{\hfill \vrule height -6pt width 0.4pt depth0pt}}}

\DeclareMathOperator{\diverge}{div}
\DeclareMathOperator{\supp}{supp}

\newtheorem{lem}{Lemma}[section]
\newtheorem{cor}[lem]{Corollary}
\newtheorem{prop}[lem]{Proposition}
\newtheorem{thm}[lem]{Theorem}
\newtheorem{remark}[lem]{Remark}

\newtheorem{definition}[lem]{Definition}

\title{Compressible, inviscid Rayleigh-Taylor instability}
\author{Yan Guo\footnote{Supported in part by NSF grant DMS-0905255 and Chinese NSF grant 10828103.}\, and Ian Tice\footnote{Supported by an NSF
Postdoctoral Research Fellowship}\\
{\small Brown University, Division of Applied Mathematics}\\
{\small 182 George St., Providence, RI 02912}\\
{\small\tt guoy@dam.brown.edu, tice@dam.brown.edu} }
\date{November 20, 2009}

\begin{document}

\maketitle

\begin{abstract}
We consider the Rayleigh-Taylor problem for two compressible, immiscible, inviscid, barotropic fluids evolving with a free interface in the presence of a uniform gravitational field.  After constructing Rayleigh-Taylor steady-state solutions with a denser fluid lying above the free interface with the second fluid, we turn to an analysis of the equations obtained from linearizing around such a steady state.  By a natural variational approach, we construct normal mode solutions that grow exponentially in time with rate like $e^{t \sqrt{\abs{\xi}}}$, where $\xi$ is the spatial frequency of the normal mode.  A Fourier synthesis of these normal mode solutions allows us to construct solutions that grow arbitrarily quickly in the Sobolev space $H^k$, which leads to an ill-posedness result for the linearized problem.  Using these pathological solutions, we then demonstrate ill-posedness for the original non-linear problem in an appropriate sense.  More precisely, we use a contradiction argument to show that the non-linear problem does not admit reasonable estimates of solutions for small time in terms of the initial data.
\end{abstract}

\section{Introduction}

\subsection{Formulation in Eulerian coordinates}

This paper concerns the compressible, inviscid Rayleigh-Taylor problem in an infinite slab $\Omega:= \Rn{2} \times (-m,\ell)\subset \Rn{3}$.  For this we consider two distinct, immiscible, inviscid, compressible, barotropic fluids evolving within $\Omega$ for time $t\ge 0$.  The fluids are separated from one another by a moving free boundary surface $\Sigma(t)$ that extends to infinity in every horizontal direction; this surface divides  $\Omega$ into two time-dependent, disjoint, open subsets $\Omega_\pm(t)$ so that $\Omega =  \Omega_+(t) \sqcup  \Omega_-(t) \sqcup \Sigma(t)$ and $\Sigma(t) = \bar{\Omega}_+(t) \cap \bar{\Omega}_-(t)$.  The fluid occupying $\Omega_+(t)$ is called the ``upper fluid,'' and the second fluid, which occupies $\Omega_-(t)$, is called the ``lower fluid.''  The two fluids are described by their density and velocity functions, which are given for each $t\ge 0$ by 
\begin{equation}
  \rho_\pm(\cdot,t) :\Omega_\pm(t) \rightarrow \Rn{+}=(0,\infty) \text{ and } u_\pm(\cdot,t) :\Omega_\pm(t) \rightarrow \Rn{3}
\end{equation}
respectively.  We shall assume that at a given time $t\ge 0$ the density and velocity functions have well-defined traces onto $\Sigma(t)$.

We require that the fluids be sufficiently smooth to satisfy the pair of compressible Euler equations:
\begin{equation}\label{eulerian_equations}
 \begin{cases}
  \dt \rho_\pm + \diverge(\rho_\pm u_\pm) = 0  & \text{for } t>0,  x \in \Omega_\pm(t) \\
  \rho_\pm (\dt u_\pm + u_\pm \cdot \nab u_\pm) + \nab(P_\pm(\rho_\pm)) = -g \rho_\pm e_3 & \text{for } t>0,  x \in \Omega_\pm(t).   
 \end{cases}
\end{equation}
Here we have written $g>0$ for the gravitational constant, $e_3 = (0,0,1)$ for the vertical unit vector, and $-ge_3$ for the acceleration due to gravity.   We have assumed a general barotropic pressure law of the form $P_\pm =  P_\pm(\rho)\ge 0$ with  $P_\pm\in C^\infty((0,\infty))$ and strictly increasing.  We will also assume that  $1/P'_{\pm}\in L^\infty_{loc}((0,\infty))$.  Finally, in order to produce the Rayleigh-Taylor instability, i.e. construct a steady state solution with an upper fluid of greater density at $\Sigma(t)$, we will assume that 
\begin{equation}\label{Z_def}
 Z:=\{z\in(0,\infty) \;\vert\; P_-(z) > P_+(z) \text{ and } P_-(z) \in P_+((0,\infty))   \} \neq \varnothing.
\end{equation}
In particular this requires the pressure laws to be distinct, i.e. $P_- \neq P_+$.  For a general discussion of the physics related to these equations and to the Rayleigh-Taylor instability, we refer to \cite{kull} and the references therein.

A standard assumption is that both the normal component of the velocity and pressure must be continuous across a free boundary between two inviscid, immiscible fluids \cite{we_la}.  This requires us to enforce the jump conditions  
\begin{equation}\label{eulerian_jumps}
 \begin{cases}
   (\nu \cdot u_+)\vert_{\Sigma(t)} -  (\nu \cdot u_-)\vert_{\Sigma(t)}  = 0 \\
   (P_+(\rho_+))\vert_{\Sigma(t)} - (P_-(\rho_-))\vert_{\Sigma(t)}  =0,
 \end{cases}
\end{equation}
where we have written the normal vector to $\Sigma(t)$ as $\nu$ and $f\vert_{\Sigma(t)}$ for the trace of a quantity $f$ on $\Sigma(t)$.  We will also enforce the condition that the normal component of the fluid velocity vanishes at the fixed upper and lower boundaries; we implement this via the boundary condition 
\begin{equation}
 u_-(x_1,x_2,-m,t) \cdot e_3 = u_+(x_1,x_2,\ell,t) \cdot e_3 =0 \text{ for all } (x_1,x_2)\in \Rn{2}, t \ge 0.
\end{equation}

The motion of the free interface is coupled to the evolution equations for the fluids \eqref{eulerian_equations} by requiring that the surface be advected with the fluids.  This means that the velocity of the surface is given by $(u \cdot \nu) \nu$.  Since the normal component of the velocity is continuous across the surface there is no ambiguity in writing $u \cdot \nu$.  The tangential components of $u_\pm$ need not be continuous across $\Sigma(t)$, and indeed there may be jumps in these.  This allows for the possibility of slipping: the upper and lower fluids moving in different directions tangent to $\Sigma(t)$.   Since only the normal component of the velocity vanishes at the fixed upper and lower boundaries, $\{x_3 = \ell\}$ and $\{x_3 = -m\}$, the fluids may also slip along the fixed boundaries.

To complete the statement of the problem, we must specify initial conditions.  We give the initial interface $\Sigma_0$, which yields the open sets $\Omega_\pm(0)$ on which we specify the initial data for the density and velocity, $\rho_\pm(0):\Omega_\pm(0) \rightarrow \Rn{+}$ and $u_\pm(0):\Omega_\pm(0) \rightarrow \Rn{3}$, respectively.

It will be convenient in our subsequent analysis to rewrite the second Euler equation using the enthalpy function 
\begin{equation}
 h_\pm(z) = \int_1^z \frac{P_\pm'(r)}{r}dr.
\end{equation}
The properties of $P_\pm$ guarantee that $h_\pm \in C^\infty((0,\infty))$ are both strictly increasing, and hence invertible on their images.  Then the second compressible Euler equation may be rewritten as
\begin{equation}
 \dt u_\pm + u_\pm \cdot \nab u_\pm + \nab(h_\pm(\rho_\pm)) = -g e_3.
\end{equation}

\subsection{Steady-state solution}

We seek a steady-state solution with $u_\pm=0$ and the interface given by $\{x_3=0\}$ for all $t\ge 0$.  Then $\Omega_+ = \Omega_+(t) = \Rn{2} \times (0,\ell), \Omega_- = \Omega_-(t) = \Rn{2} \times(-m,0)$ for all $t\ge 0$, and the equations reduce to the ODE
\begin{equation}\label{enthalpy_eqn}
 \frac{d(h_\pm(\rho_\pm))}{dx_3} = -g  \text{ in }\Omega_\pm
\end{equation}
subject to the jump condition
\begin{equation}
 P_+(\rho_+) = P_-(\rho_-) \text{ on } \{x_3=0\}.
\end{equation}
Such a solution depends only on $x_3$, so we may consolidate notation by assuming that $\rho_\pm$ are the restrictions to $(0,\ell)$ and $(-m,0)$ of a single function $\rho_0 = \rho_0(x_3)$ that is smooth on $(-m,0)$ and $(0,\ell)$ with a  jump discontinuity across $x_3=0$.

To solve this we recall that $h_\pm \in C^\infty((0,\infty))$ are both strictly increasing and invertible on their images.  The solution to the ODE is then given by
\begin{equation}
 \rho_0(x_3) = \begin{cases}
h_-^{-1}(h_-(\rho^-_0) - g x_3 ), & -m < x_3 < 0 \\
h_+^{-1}(h_+(\rho^+_0) - g x_3 ), &  0< x_3 < \ell,
             \end{cases}
\end{equation}
where $\rho^-_0 >0$ is a free parameter satisfying $P_-(\rho_0^-) \in P_+((0,\infty))$, which allows the jump condition to be satisfied by choosing $\rho_0^+>0$ according to
\begin{equation}\label{pressure_jump}
 \rho^+_0 = P_+^{-1}(P_-(\rho^-_0)).
\end{equation}
For $\rho_0$ to be well-defined,  we will henceforth assume that $\ell, m>0$ are chosen so that
\begin{equation}
(h_-(\rho^-_0) + g m) \in h_-((0,\infty)) \text{ and } (h_+(\rho^+_0) - g \ell) \in h_+((0,\infty)).
\end{equation}
Note that $\rho_0$ is bounded above and below by positive constants on $(-m,\ell)$ and that $\rho_0$ is smooth when restricted to $(-m,0)$ or $(0,\ell).$

Since we are interested in Rayleigh-Taylor instability, we want the upper fluid to be denser than the lower fluid at the interface, i.e. $\rho^+_0 > \rho^-_0$.  This requires us to choose $\rho^-_0$ so that
\begin{equation}
P_+^{-1}(P_-(\rho^-_0)) > \rho^-_0 \Leftrightarrow P_-(\rho^-_0) > P_+(\rho^-_0).
\end{equation}
The latter condition is satisfied for any $\rho^-_0 \in Z$, where $Z$ was defined by \eqref{Z_def}; we assume $\rho^-_0$ takes any such value.  Then
\begin{equation}\label{rho_jump}
 \jump{\rho_0} := \rho_0^+ - \rho_0^- >0.
\end{equation}

For the sake of clarity, we include an example of the solution, $\rho_0$, when the pressure laws correspond to polytropic gas laws, i.e. $P_\pm(\rho) = K_\pm \rho^{\gamma_\pm}$ for $K_\pm> 0, \gamma_\pm \ge 1$.  The solution is then given by
\begin{equation}
 \rho_0(x_3) = \begin{cases}
\left((\rho_0^-)^{\gamma_{-}-1} - \frac{g(\gamma_{-}-1)}{K_{-}\gamma_{-}} x_3   \right)^{1/(\gamma_{-}-1)} & x_3 < 0 \\
\left((\rho_0^+)^{\gamma_{+}-1} - \frac{g(\gamma_{+}-1)}{K_{+}\gamma_{+}} x_3   \right)^{1/(\gamma_{+}-1)} &
0<x_3 \le \frac{K_{+}\gamma_{+}}{g(\gamma_{+}-1)}  (\rho_0^+)^{\gamma_{+}-1}     \\
0 & x_3 \ge \frac{K_{+}\gamma_{+}}{g(\gamma_{+}-1)}  (\rho_0^+)^{\gamma_{+}-1} 
\end{cases}
\end{equation}
with modification to solutions $\rho_0(x_3) = \rho_0^{\pm} \exp(-gx_3/K_{\pm})$ when either $\gamma_+$ or $\gamma_-$ is $1$.  The jump condition requires that 
\begin{equation}
 \rho_0^+ = \left(\frac{K_-}{K_+}\right)^{1/\gamma_+} (\rho_0^-)^{\gamma_-/\gamma_+}.
\end{equation}
For a polytropic gas law, the condition that $\rho_0^+>\rho_0^-$ is equivalent to
\begin{equation}
 \left(\frac{K_-}{K_+}\right)^{1/\gamma_+} (\rho_0^-)^{\gamma_-/\gamma_+} > \rho_0^- \Leftrightarrow (\rho_0^-)^{\gamma_- - \gamma_+} > \frac{K_+}{K_-}.
\end{equation}
If $\gamma_+ = \gamma_-$ this requires $K_- > K_+$ and any choice of $\rho_0^->0$.  If $\gamma_+ \neq \gamma_-$ then $K_-, K_+>0$ can be arbitrary, but we must require that $\rho_0^->0$ satisfies
\begin{equation}
\begin{cases}
 \rho_0^- > \left( \frac{K_+}{K_-}\right)^{1/(\gamma_- - \gamma_+)} & \text{if } \gamma_- > \gamma_+\\
 \rho_0^- < \left( \frac{K_-}{K_+}\right)^{1/(\gamma_+ - \gamma_-)} & \text{if } \gamma_+ > \gamma_-.
\end{cases}
\end{equation}
In either case, to avoid the vanishing of $\rho_0$, $\ell$ is chosen so that
\begin{equation}
0< \ell < \frac{K_{+}\gamma_{+}}{g(\gamma_{+}-1)}  (\rho_0^+)^{\gamma_{+}-1},
\end{equation}
but the parameter $m>0$ may be chosen arbitrarily.

\subsection{Horizontal Fourier transform and piecewise Sobolev spaces}

Before stating the main results, we define some terms that will be used throughout the paper.  For a function $f\in L^2(\Omega)$, we define the horizontal Fourier transform via 
\begin{equation}\label{hft}
 \hat{f}(\xi_1,\xi_2,x_3) = \int_{\Rn{2}} f(x_1,x_2,x_3) e^{-i(x_1 \xi_1 + x_2 \xi_2)}dx_1 dx_2.
\end{equation}
By the Fubini and Parseval theorems, we have that 
\begin{equation}\label{parseval}
\int_\Omega \abs{f(x)}^2 dx = \frac{1}{4\pi^2} \int_\Omega \abs{\hat{f}(\xi,x_3)}^2 d\xi dx_3.
\end{equation}

We now define a function space suitable for our analysis of two disjoint fluids.  For a function $f$ defined on $\Omega$ we write $f_+$ for the restriction to $\Omega_+=\Rn{2} \times (0,\ell)$ and $f_-$ for the restriction to $\Omega_-=\Rn{2} \times (-m,0)$.  For $s \in \Rn{}$, define the piecewise Sobolev space of order $s$ by 
\begin{equation}\label{sob_def}
H^s(\Omega) = \{f \;\vert\; f_+ \in H^s(\Omega_+), f_- \in H^s(\Omega_-)  \}
\end{equation} 
endowed with the norm $\norm{f}_{H^s}^2 = \norm{f}_{H^s(\Omega_+)}^2 + \norm{f}_{H^s(\Omega_-)}^2$.  For $k\in \mathbb{N}$ we can take the norms to be given as
\begin{multline}
 \norm{f}_{H^k(\Omega_\pm)}^2 := \sum_{j=0}^k  \int_{\Rn{2}\times I_\pm} (1+\abs{\xi}^2)^{k-j} \abs{\partial_{x_3}^j \hat{f_\pm}(\xi,x_3) }^2 d\xi dx_3 \\
=  \sum_{j=0}^k  \int_{\Rn{2}} (1+\abs{\xi}^2)^{k-j} \norm{\partial_{x_3}^j \hat{f}_\pm(\xi,\cdot) }^2_{L^2(I_\pm)} d\xi 
\end{multline}
for $I_- = (-m,0)$ and $I_+ = (0,\ell)$.  The main difference between the piecewise Sobolev space $H^s(\Omega)$ and the usual Sobolev space is that we do not require functions in the piecewise Sobolev space to have weak derivatives across the set $\{x_3=0\}$.

\subsection{Summary of main results and plan of paper}

The motion of the free surface $\Sigma(t)$ and the domains $\Omega_\pm(t)$ present several mathematical difficulties, so we begin our analysis by switching to a Lagrangian coordinate system in which the domains of the upper and lower fluids stay fixed in time as $\Omega_+ = \Rn{2} \times (0,\ell)$ and $\Omega_-= \Rn{2} \times (-m,0)$, respectively.  Since the steady-state solution has vanishing fluid velocity, the steady-state constructed above in Eulerian coordinates is also a steady-state in Lagrangian coordinates.   

The first part of the paper is devoted to a study of the equations obtained by linearizing the compressible Euler equations, written in Lagrangian coordinates, around the steady-state solution.  The resulting linear equations have coefficient functions that depend only on the vertical variable, $x_3\in(-m,\ell)$.  This allows us to seek ``normal mode'' solutions (cf. \cite{chandra}) by taking the horizontal Fourier transform of the equations and assuming the solution grows exponentially in time by the factor $e^{\lambda(\xi)t}$, where $\xi \in \Rn{2}$ is the horizontal spatial frequency and $\lambda(\xi)>0$.  A similar strategy was employed in the non-barotropic, horizontally periodic case in \cite{vander}.  This reduces the equations to a system of ODEs for each $\xi$, which constitute a variant of a classical boundary-value eigenvalue problem with eigenvalue $\lambda(\xi)$ since one of the unknown functions in the system only appears with first order derivatives in the equations.  

In spite of its non-standard structure, the eigenvalue problem is amenable to solution by constrained minimization, which we then employ.  The use of variational methods to produce the normal mode solution for each $\xi$ is essential to our analysis since it gives rise to detailed estimates of the behavior of $\lambda(\xi)$ as $\xi$ varies in $\Rn{2}$.  We show  in Lemma \ref{eigen_lower}  that there is a constant $C$ so that
\begin{equation}
 0 < C \le  \frac{\lambda(\xi)}{\sqrt{\abs{\xi}}} \le \sqrt{g}   \text{ as } \abs{\xi}\rightarrow \infty.
\end{equation}
Since $\lambda(\xi) \rightarrow \infty$, normal modes with a higher spatial frequency grow faster in time, which provides a mechanism for Rayleigh-Taylor instability.  Indeed, we can form a Fourier synthesis of the normal mode solutions constructed for each spatial frequency $\xi$ to construct solutions to the linearized compressible Euler equations that  grow arbitrarily quickly in time, when measured in $H^k(\Omega)$ for any $k \ge 0$.  This is the content of Theorem \ref{growing_mode_soln}.  The variational methods used are quite general and robust; indeed, in \cite{tice} we extend the methods of this paper to construct growing solutions to the linearized Navier-Stokes equations with surface tension and density-dependent viscosity coefficients.

As a preliminary for our analysis of the well-posedness of the non-linear problem, we then consider the well-posedness of the linearized problem.  Inspired by a result in \cite{hw_guo}, we show a connection between the growth rate of arbitrary solutions to the linearized equations and the eigenvalues $\lambda(\xi)$, which then gives rise to a uniqueness result, Theorem \ref{linear_uniqueness}.  In spite of the uniqueness, the linear problem is ill-posed in the sense of Hadamard in $H^k(\Omega)$ for any $k$ since solutions do not depend continuously on the initial data.  This is shown in Theorem \ref{linear_ill_posed} by employing  Theorem \ref{growing_mode_soln} to build a sequence of solutions with initial data tending to $0$ in $H^k(\Omega)$, but which grow to be arbitrarily large in $H^k(\Omega)$ arbitrarily quickly.  Again, the construction depends heavily on the detailed knowledge of the normal mode solutions provided by the variational methods.

With linear ill-posedness established, we then prove ill-posedness of the fully non-linear problem.  There is no general theory that guarantees the ill-posedness of a non-linear problem given the ill-posedness of the resulting linearized problem.  As such, it is novel and remarkable that in the present case, linear ill-posedness does indeed give rise to  ill-posedness for the non-linear problem.  This is in some sense a compressible analogue to the Rayleigh-Taylor ill-posedness results in the incompressible regime, proved in \cite{ebin}. 

To see ill-posedness, we rewrite the equations in Lagrangian coordinates with the unknown functions given as perturbations of the steady-state solution (see \eqref{perturb} for the exact formulation).  For any $k\ge 3$, we then define a notion of non-linear well-posedness, which we call property $EE(k)$ (see Definition \ref{EE_def} for the precise statement).   Property $EE(k)$ requires local-in-time (on an interval $(0,t_0)$) existence of solutions for initial data with small $H^k(\Omega)$ norm, along with $L^\infty((0,t_0);H^3(\Omega))$ estimates of the solutions (written here generically as $X(t)$) of the form
\begin{equation}
 \sup_{0\le t \le t_0}\norm{X(t)}_{H^3(\Omega)} \le F(\norm{X(0)}_{H^k(\Omega)})
\end{equation}
for some  function $F$, satisfying the Lipschitz condition $F(z) \le C z$ for some $C>0$ for all $z$ in a neighborhood of $0$.   Condition $EE(k)$ is quite general and a reasonable choice for any well-posedness theory.

We then show in Theorem \ref{no_wpk} that it is impossible for property $EE(k)$ to hold for any $k\ge 3$, which implies  ill-posedness for the non-linear compressible Euler equations.  The important feature of $EE(k)$ is that $k\ge 3$ is arbitrary.  If the initial data are extremely smooth ($k$ very large), the failure of property $EE(k)$ means it is impossible to control even the $H^3(\Omega)$ norm for small time.  The proof of Theorem \ref{no_wpk} is a reductio ad absurdum.  Indeed, we show that if $EE(k)$ holds, then it is possible to obtain certain estimates for the corresponding linearized equations that violate our linear ill-posedness result, Theorem \ref{linear_ill_posed}. 

All of our analysis is performed on the semi-infinite slab $\Omega= \Rn{2} \times (-m,\ell)$, and one may wonder if the Rayleigh-Taylor instability is somehow related to the infinite extent of the horizontal component $\Rn{2}$.  This is \emph{not} the case; all of our results carry over essentially word for word if $\Omega$ is replaced by $\tilde{\Omega} = \mathbb{T}^2 \times (-m,\ell)$, where $\mathbb{T}^2 = \mathbb{S}^1 \times \mathbb{S}^1$ is the $2-$torus.  In this case, all of the functions are required to be periodic in the horizontal directions, and the horizontal Fourier transform  \eqref{hft} with continuous spatial frequencies $\xi \in \Rn{2}$  must be replaced with the horizontal Fourier transform of $\mathbb{T}^2$, for which the spatial frequencies are constrained to $\xi \in \mathbb{Z}^2$.  For the sake of brevity we will not rewrite the results for $\tilde{\Omega}$.  There is also nothing essential in our analysis about the domain being contained in $\Rn{3}$.  The same results hold if $\Omega$ or $\tilde{\Omega}$ is replaced by $\Omega = \Rn{1} \times (-m,\ell)$ or $\tilde{\Omega} = \mathbb{T}^1 \times (-m,\ell)$.

The paper is organized as follows.  In Section 2 we perform the switch to Lagrangian coordinates and record the linearized equations.  In Section 3 we construct the growing solutions to the linearized equations.  In Section 4 we analyze the linear problem, proving uniqueness and discontinuous dependence on the initial data.  In Section 5 we prove the ill-posedness result for the non-linear problem.

\section{Formulation in Lagrangian coordinates}

\subsection{Lagrangian coordinates}

The movement of the free boundary and the subsequent change of the domains $\Omega_\pm(t)$ in Eulerian coordinates create numerous mathematical difficulties.  We circumvent these difficulties by switching to Lagrangian coordinates so that the free interface and the domains stay fixed in time.  To this end we define the fixed Lagrangian domains $\Omega_- = \Rn{2} \times (-m,0)$ and $\Omega_+ = \Rn{2} \times (0,\ell)$ and assume that there exist invertible mappings 
\begin{equation}
 \eta^0_\pm:\Omega_\pm \rightarrow \Omega_\pm(0)
\end{equation}
so that $\Sigma_0 = \eta^0_+(\{x_3=0\})$, $\eta^0_+(\{x_3=\ell\}) = \{x_3=\ell\}$, and $\eta^0_-(\{x_3=-m \}) = \{x_3=-m\}$.  The first condition means that $\Sigma_0$ is parameterized by the mapping $\eta^0_+$ restricted to $\Rn{2}\times\{0\}$, and the latter two conditions mean that $\eta_\pm^0$ map the fixed upper and lower boundaries into themselves.  Define the flow maps $\eta_\pm$ as the solution to
\begin{equation}
 \begin{cases}
  \dt \eta_\pm(x,t) = u_\pm(\eta_\pm(x,t),t) \\
  \eta(x,0) = \eta_\pm^0(x).
 \end{cases}
\end{equation}
We think of the Eulerian coordinates as $(y,t)$ with $y=\eta(x,t)$, whereas we think of Lagrangian coordinates as the fixed $(x,t)\in \Omega \times \Rn{+}$; this implies that $\Omega_\pm(t) = \eta_\pm(\Omega_\pm,t)$ and that $\Sigma(t) = \eta_+(\{x_3=0\},t)$, i.e. that the Eulerian domains of upper and lower fluids are the image of $\Omega_\pm$ under the mappings $\eta_\pm$ and that the free interface is parameterized by $\eta_+(\cdot,t)$ restricted to $\Rn{2} \times \{0\}$.  In order to switch back and forth from Lagrangian to Eulerian coordinates we assume that $\eta_\pm(\cdot,t)$ is invertible.  Since the upper and lower fluids may slip across one another, we must introduce the slip map $S_{\pm}:\Rn{2} \times \Rn{+} \rightarrow \Rn{2} \times \{0\} \subset \Rn{2} \times (-m,\ell)$ defined by
\begin{equation}\label{slip_map_def}
 S_-(x_1,x_2,t) = \eta_-^{-1}(\eta_+(x_1,x_2,0,t),t)
\end{equation}
and $S_+(\cdot,t) = S_-^{-1}(\cdot,t)$.  The slip map $S_-$ gives the particle in the lower fluid  that is in contact with the particle of the upper fluid at $x = (x_1,x_2,0)$ on the contact surface at time $t$.

We define the Lagrangian unknowns
\begin{equation}
 \begin{cases}
  v_\pm(x,t) = u_\pm(\eta_\pm(x,t),t) \\
  q_\pm(x,t) = \rho_\pm(\eta_\pm(x,t),t),
 \end{cases}
\end{equation}
both of which are defined  for $(x,t) \in \Omega_\pm \times \Rn{+}$.  Since the domains $\Omega_\pm$ are now fixed, we henceforth consolidate notation by writing $\eta, v, q$ to refer to $\eta_\pm, v_\pm, q_\pm$ except when necessary to distinguish the two; when we write an equation for $\eta, v, q$ we assume that the equation holds with the subscripts added on the domains $\Omega_\pm$.  Define the matrix $A$ via $A^T= (D \eta)^{-1}$, where $D$ is the derivative  in $x$ coordinates and superscript $T$ denotes matrix transposition.  Then in Lagrangian coordinates, the equations for $v,q,\eta$ are, writing $\partial_j = \partial / \partial x_j$, $\nab = (\partial_1,\partial_2,\partial_3)$, and tr$(\cdot)$ for the matrix trace, 
\begin{equation}\label{lagrangian_equations}
 \begin{cases}
  \dt \eta = v \\
  \dt q + q \text{tr}(A Dv) =0 \\
  \dt v + A \nab(h(q)) = -g A \nab \eta_3.
 \end{cases}
\end{equation}

Since the boundary jump conditions in Eulerian coordinates are phrased in terms of jumps across the surface, the slip map must be employed in Lagrangian coordinates.  The  jump conditions in Lagrangian coordinates are
\begin{equation}\label{lagrangian_jumps}
 \begin{cases}
  (v_+(x_1,x_2,0,t) - v_-(S_-(x_1,x_2,t),t) )\cdot n(x_1,x_2,0,t)  =0 \\
  P_+(q_+(x_1,x_2,0,t)) = P_-(q_-(S_-(x_1,x_2,t),t))
 \end{cases}
\end{equation}
where we have written $n= \nu \circ \eta $, i.e. 
\begin{equation}
n:= \frac{\partial_1 \eta_+ \times \partial_2 \eta_+}{\abs{\partial_1 \eta_+ \times \partial_2 \eta_+}} 
\end{equation}
for the normal to the surface $\Sigma(t) = \eta_+(\{x_3=0\},t)$.  Note that we could just as well have phrased the jump conditions in terms of the slip map $S_+$ and defined the surface and its normal vector in terms of $\eta_-$.  Finally, we require 
\begin{equation}
 v_-(x_1,x_2,-m,t) \cdot e_3 = v_+(x_1,x_2,\ell,t) \cdot e_3 =0.
\end{equation}
Note that since $\dt \eta = v$, 
\begin{equation}
e_3 \cdot \eta_+(x_1,x_2,\ell,t) =e_3 \cdot \eta_+^0(x_1,x_2,\ell) + \int_0^t e_3 \cdot v_+(x_1,x_2,\ell,s)ds = \ell,
\end{equation}
which implies that $\eta_+(x_1,x_2,\ell,t) \in \{x_3 = \ell\}$ for all $t\ge 0$, i.e. that the part of the upper fluid in contact with the fixed boundary $\{x_3=\ell\}$ never flows down from the boundary.  It may, however, slip along the fixed boundary since we do not require $v_+(x_1,x_2,\ell,t)\cdot e_i = 0$ for $i=1,2$.  A similar result holds for $\eta_-$ at the lower fixed boundary $\{x_3=-m\}$. 

In the subsequent analysis it will be convenient to employ the notation
\begin{equation}
 \jump{f} = f_+ \vert_{\{x_3 =0\}} - f_-\vert_{\{x_3 =0\}}
\end{equation}
for the jump of a quantity $f$ across the set $\{x_3=0\}$.

\subsection{Linearization in Lagrangian coordinates}

In the steady-state case, the flow map is the identity mapping, $\eta = Id$, so that $v=u$ and $q = \rho$.  This means that the steady-state solution, $\rho_0$, constructed above in Eulerian coordinates is also a steady state in Lagrangian coordinates.  Now we want to linearize the equations around the steady-state solution $v=0$, $\eta = Id$, $q=\rho_0$, for which  $S_- = Id_{\{x_3=0\}}$ and $A=I$, the $3\times 3$ identity matrix.  The resulting linearized equations are
\begin{equation}\label{linearized}
\begin{cases}
 \dt \eta = v \\
 \dt q + \rho_0 \diverge{v} =0 \\
 \rho_0 \dt v + \nab(P'(\rho_0) q) =  - g q e_3 - g\rho_0 \nab(e_3 \cdot \eta).
\end{cases}
\end{equation}
Here we have rewritten the third equation by utilizing \eqref{enthalpy_eqn} to write
\begin{equation}
 \nab(h'(\rho_0) q) = \frac{1}{\rho_0}\nab(P'(\rho_0)q) - \frac{P'(\rho_0)\rho_0' q}{\rho_0^2} e_3 = \frac{1}{\rho_0}\nab(P'(\rho_0)q) + \frac{g q e_3}{\rho_0}
\end{equation}
and then multiplying by $\rho_0$, which does not vanish on $(-m,\ell)$.

The jump conditions linearize to
\begin{equation}
 \jump{v \cdot e_3 }   =0 \text{ and }  \jump{P'(\rho_0) q}  =0, 
\end{equation}
while the boundary conditions linearize to
\begin{equation}
 v_-(x_1,x_2,-m,t) \cdot e_3 = v_+(x_1,x_2,\ell,t) \cdot e_3 =0.
\end{equation}

\section{Construction of a growing solution to \eqref{linearized}}

\subsection{Growing mode ansatz}

We wish to construct a solution to the linearized equations \eqref{linearized} that has a growing $H^k$ norm for any $k$.  We will construct such solutions via Fourier synthesis by first constructing a growing mode for a fixed spatial frequency.

To begin, we assume a growing mode ansatz, i.e. let us assume that
\begin{equation}
 v(x,t) = w(x) e^{\lambda t}, q(x,t)= \tilde{q}(x) e^{\lambda t}, \eta(x,t) = \tilde{\eta}(x) e^{\lambda t}.
\end{equation}
Here we assume that $\lambda\in \mathbb{C}\backslash\{0\}$ is the same above and below the interface.  A solution with $\Re(\lambda)>0$ corresponds to a growing mode.  Plugging the ansatz into \eqref{linearized}, we get the equations
\begin{equation}\label{growing_mode}
 \begin{cases}
  \lambda \tilde{\eta} = w \\
  \lambda \tilde{q} + \rho_0 \diverge{w} = 0 \\
 \lambda \rho_0  w + \nab(P'(\rho_0) \tilde{q}) =  - g \tilde{q} e_3 - g\rho_0 \nab(e_3 \cdot \tilde{\eta}).
 \end{cases}
\end{equation}
Eliminating the unknowns $\tilde{\eta}$ and $\tilde{q}$ by using the first and second equations, we arrive at the time-invariant system
\begin{equation}\label{linear_timeless}
 \lambda^2 \rho_0  w  - \nab\left(P'(\rho_0)\rho_0 \diverge{w} \right)    \\
=  g\rho_0 \diverge{w}  e_3  -g\rho_0 \nab w_3.
\end{equation}
We may also eliminate $\tilde{q}$ in the jump conditions to get
\begin{equation}
\jump{w_3} =0 \text{ and } \jump{ P'(\rho_0) \rho_0 \diverge{w}   }  =   0.
\end{equation}
Here we have used the fact that $\lambda \neq 0$ to remove it from the second jump condition.  The boundary conditions are
\begin{equation}
 w_3(x_1,x_2,-m) = w_3(x_1,x_2,\ell)  =0.
\end{equation}

\subsection{Horizontal Fourier transformation}
Since the coefficients of the linear problem \eqref{linear_timeless} only depend on the $x_3$ variable, we are free to take the horizontal Fourier transform \eqref{hft}, which we denote with either $\hat{\cdot}$ or $\mathcal{F}$, to reduce to a system of ODEs in $x_3$ for each fixed spatial frequency.  

We take the horizontal Fourier transform of $w_1,w_2,w_3$ in \eqref{linear_timeless} and fix a spatial frequency $\xi =(\xi_1,\xi_2)\in\Rn{2}$.  Define the new unknowns $\varphi(x_3)=i \hat{w}_1(\xi_1,\xi_2,x_3)$, $\theta(x_3)=i \hat{w}_2(\xi_1,\xi_2,x_3)$, and $\psi(x_3)=\hat{w}_3(\xi_1,\xi_2,x_3)$ so that 
\begin{equation}
 \mathcal{F} (\diverge{w}) = \xi_1 \varphi + \xi_2 \theta + \psi',
\end{equation}
where $' = d/dx_3$.  Then for $\varphi,\theta,\psi$ and $\lambda = \lambda(\xi)$ we arrive at the following system of ODEs.
\begin{equation}\label{w_1_equation}
\left( \lambda^2 \rho_0  +  \xi_1^2  P'(\rho_0)\rho_0    \right) \varphi  = - \xi_1 \left(  P'(\rho_0)\rho_0  \psi' - g\rho_0 \psi \right) 
 - \xi_1  \xi_2  P'(\rho_0)\rho_0 \theta 
\end{equation}
\begin{equation}\label{w_2_equation}
\left( \lambda^2 \rho_0   + \xi_2^2  P'(\rho_0)\rho_0    \right) \theta  = - \xi_2 \left(  P'(\rho_0)\rho_0   \psi' - g\rho_0 \psi \right) 
 -\xi_1   \xi_2   P'(\rho_0)\rho_0  \varphi 
\end{equation}
\begin{equation}\label{w_3_equation}
 -\left(P'(\rho_0)\rho_0  \psi'\right)' 
+  \lambda^2 \rho_0   \psi 
=  \xi_1 \left(\left(  P'(\rho_0)\rho_0  \varphi \right)' + g\rho_0 \varphi \right)
+  \xi_2 \left(\left(  P'(\rho_0)\rho_0  \theta \right)' + g\rho_0 \theta \right)
\end{equation}

We can reduce the complexity of the problem by removing the component $\theta$.  To do this, note that if $\varphi,\theta,\psi$ solve the above equations for $\xi_1,\xi_2$ and $\lambda$, then for any rotation operator $R\in SO(2)$,  $(\tilde{\varphi},\tilde{\theta}) := R (\varphi,\theta)$ solve the same equations for $(\tilde{\xi}_1,\tilde{\xi}_2) := R(\xi_1,\xi_2)$ with $\psi, \lambda$ unchanged.  So, by choosing an appropriate rotation, we may assume without loss of generality that $\xi_2=0$ and $\xi_1=\abs{\xi}\ge 0$.  In this setting $\theta$ solves 
\begin{equation}
\lambda^2 \rho_0  \theta =0, 
\end{equation}
and since  $\rho_0>0$ on $(-m,\ell)$ we have that $\theta=0$.  This reduces to the pair of equations
\begin{equation}\label{coupled}
\begin{cases}
 -\lambda^2 \rho_0 \varphi  =  \abs{\xi} ( P'(\rho_0)\rho_0  (\psi' + \abs{\xi} \varphi)) -  g\abs{\xi}\rho_0 \psi  \\
 -\lambda^2 \rho_0  \psi =  - (P'(\rho_0)\rho_0 (\psi' + \abs{\xi} \varphi )) ' - g \abs{\xi} \rho_0 \varphi 
\end{cases}
\end{equation}
with the jump conditions
\begin{equation}
\jump{ \psi }     = 0 \text{ and } \jump{  P'(\rho_0)\rho_0 (\abs{\xi} \varphi + \psi')}  = 0
\end{equation}
and the boundary conditions $\psi(-m) = \psi(\ell) =0$.

Since $\lambda \in \mathbb{C}$,  $\varphi,\psi$ can be complex valued, but we may reduce to the case of real valued functions as follows.  Multiply the first equation in \eqref{coupled} by $\bar{\varphi}$ and the second by $\bar{\psi}$, add the resulting equations, and integrate over $(-m,\ell)$.  An integration by parts and an application of the boundary and jump conditions shows that
\begin{equation}
 -\frac{\lambda^2}{2} \int_{-m}^\ell \rho_0 (\abs{\varphi}^2 + \abs{\psi}^2) = \hal \int_{-m}^\ell P'(\rho_0)\rho_0 \abs{\psi' + \abs{\xi} \varphi}^2 - 2 g \abs{\xi} \rho_0 \Re(\varphi \bar{\psi}).
\end{equation}
From this we see that $-\lambda^2 \in \Rn{}$, which implies that $\Re(\varphi), \Re(\psi)$ are also solutions, so we may restrict ourselves to finding real valued solutions.  Moreover, we know that either $\lambda \in \Rn{}$ if $-\lambda^2 \le 0$ or $\lambda \in i\Rn{}$ if $-\lambda^2 >0$.  Since we want $\lambda >0$ for a growing mode solution, we are interested only in finding solutions with $-\lambda^2 := \mu < 0$.  For a given $\abs{\xi}$, the smallest value of $-\lambda^2 =\mu$ (and hence the largest growing mode if $\mu<0$) can then be found by a variational principle:
\begin{multline}
\mu = \mu(\abs{\xi}) \\
=  \inf\left\{ \left. \hal \int_{-m}^\ell P'(\rho_0)\rho_0 (\psi' + \abs{\xi} \varphi)^2 - 2 g \abs{\xi} \rho_0 \varphi \psi \;\right\vert\; \hal \int_{-m}^\ell \rho_0 (\varphi^2 + \psi^2) =1 \right\}.
\end{multline}
We will use this in the next section to find solutions.


\subsection{Constrained minimization formulation}

In this section we will produce a solution to \eqref{coupled} with $\abs{\xi}>0$  by utilizing variational methods.  In order to understand $\lambda$ in a variational framework we consider the two energies
\begin{equation}
E(\varphi,\psi) = \hal \int_{-m}^\ell  P'(\rho_0)\rho_0 (\psi' + \abs{\xi} \varphi)^2    
- 2  g\rho_0  \abs{\xi}  \psi \varphi  
\end{equation}
and 
\begin{equation}
 J(\varphi,\psi) = \hal \int_{-m}^\ell \rho_0  (\varphi^2 + \psi^2),
\end{equation}
which are both well-defined on the space $L^2((-m,\ell)) \times H_0^1((-m,\ell))$.  In this section we will write $H_0^1((-m,\ell)$ for the usual Sobolev space of functions on $(-m,\ell)$ that vanish at the endpoints, i.e. for our variational methods we will \emph{not} use the piecewise Sobolev space defined in the introduction.  Consider the set
\begin{equation}
 \mathcal{A} = \{ (\varphi,\psi)\in  L^2((-m,\ell)) \times H_0^1((-m,\ell)) \;\vert\;  J(\varphi,\psi)=1  \}.
\end{equation}

We will generate solutions by minimizing $E$ over the collection $\mathcal{A}$.  In order for these minimizers to give rise to a solution to \eqref{coupled}, it must hold that 
\begin{equation}
 -\lambda^2 =  \inf_{(\varphi,\psi)\in \mathcal{A}} E(\varphi,\psi) < 0.
\end{equation}
Throughout this section we assume that $\abs{\xi}>0$, and we will construct a non-trivial solution with $\lambda = \lambda(\abs{\xi}) >0$.  To this end, we begin by showing that the infimum is negative for each $\abs{\xi}>0$.

\begin{lem}\label{neg_inf}
It holds that  $\inf \{ E(\varphi,\psi) \;\vert\; (\varphi,\psi) \in \mathcal{A} \} < 0.$
\end{lem}
\begin{proof}
 Since both $E$ and $J$ are homogeneous of degree $2$ it suffices to show that 
\begin{equation}
 \inf_{(\varphi,\psi)\in L^2 \times H_0^1} \frac{E(\varphi,\psi)}{J(\varphi,\psi)} < 0,
\end{equation}
but since $J$ is positive definite, we may reduce to constructing any pair $(\varphi,\psi)\in L^2 \times H_0^1$ such that $E(\varphi,\psi) <0$.  We will further assume that $\varphi = -\psi'/\abs{\xi}$ so that the first term in $E(\varphi,\psi)$ vanishes.  We must then construct $\psi \in H_0^1$ so that 
\begin{equation}
\tilde{E}(\psi) := E(-\psi'/\abs{\xi},\psi) =  \int_{-m}^\ell g \rho_0 \psi \psi' <0. 
\end{equation}

We employ the identity $\psi \psi' = (\psi^2)'/2$ and an integration by parts to write
\begin{multline}\label{n_i_1}
 \tilde{E}(\psi) = \left[\frac{g \rho_0 \psi^2 }{2}  \right]_{0}^{\ell} - \hal \int_{0}^\ell g \rho_0' \psi^2 + \left[\frac{g \rho_0 \psi^2 }{2} \right]_{-m}^{0} - \hal \int_{-m}^0 g \rho_0' \psi^2 \\
=  -\frac{g \psi^2(0)}{2} \jump{ \rho_0} + \frac{g^2}{2} \int_{-m}^\ell \frac{\rho_0}{P'(\rho_0)} \psi^2.
\end{multline}
Notice that by \eqref{rho_jump}, $\jump{\rho_0}  = \rho^+_0 - \rho^-_0 >0$  so that the right hand side is not positive definite.  

For $\alpha \ge 2$ we define the function $\psi_\alpha \in H_0^1((-m,\ell))$ according to 
\begin{equation}
 \psi_\alpha(x_3) = 
\begin{cases}
   \left(\frac{\rho^+_0 P_+'(\rho_0(x_3))}{P_+'(\rho^+_0) \rho_0(x_3)} \right)^{1/2} \left(1-\frac{x_3}{\ell} \right)^{\alpha/2}, & x_3 \in (0,\ell) \\
   \left(\frac{\rho^-_0 P_-'(\rho_0(x_3))}{P_-'(\rho^-_0) \rho_0(x_3)} \right)^{1/2}
   \left(1+\frac{x_3}{m} \right)^{\alpha/2}, &  x_3\in (-m,0]. 
\end{cases}
\end{equation}
According to \eqref{pressure_jump}, $\psi_\alpha$ is continuous across $x_3=0$ and $\psi_\alpha(0)=1$.  A simple calculation then shows that 
\begin{equation}
 \int_{-m}^\ell \frac{\rho_0}{P'(\rho_0)} \psi^2 = \frac{1}{1+\alpha}\left( \frac{m \rho^-_0}{P_-'(\rho^-_0) } +  \frac{\ell \rho^+_0}{P_+'(\rho^+_0)}  \right),
\end{equation}
which implies that
\begin{equation}\label{n_i_2}
 \tilde{E}(\psi_\alpha) = \frac{g}{2} \left( -\jump{\rho_0} +  \frac{g}{\alpha +1}\left( \frac{m \rho^-_0}{P'_-(\rho^-_0)} +  \frac{\ell \rho^+_0}{P'_+(\rho^+_0) }  \right)  \right).
\end{equation}
Then for $\alpha$ sufficiently large we have that $\tilde{E}(\psi_\alpha) < 0$, thereby proving the result.
\end{proof}

The key to the argument presented in Lemma \ref{neg_inf} was constructing a pair $(\varphi,\psi)$ so that
\begin{equation}
 \frac{-g\jump{\rho_0}}{2} \psi^2(0) <0,
\end{equation}
which in particular required that $\psi(0)\neq 0$.  We can show that this property is satisfied by any $(\varphi,\psi) \in \mathcal{A}$ so that $E(\varphi,\psi) <0$.

\begin{lem}\label{non_zero_origin}
 Suppose that $(\varphi,\psi)\in \mathcal{A}$ satisfy $E(\varphi,\psi)<0$.  Then $\psi(0)\neq 0$.
\end{lem}
\begin{proof}
A completion of the square allows us to write 
\begin{multline}
 P'(\rho_0)\rho_0 (\psi' + \abs{\xi} \varphi)^2 - 2 g \rho_0 \abs{\xi} \psi \varphi \\
= \left(\sqrt{P'(\rho_0)\rho_0}(\psi'+\abs{\xi} \varphi) - \frac{g\sqrt{\rho_0}}{\sqrt{P'(\rho_0)}}  \psi  \right)^2  + 2 g \rho_0 \psi \psi' - \frac{g^2 \rho_0}{P'(\rho_0)}  \psi^2.
\end{multline}
Integrating by parts as in \eqref{n_i_1}, we know that
\begin{equation}
  \int_{-m}^\ell 2g \rho_0 \psi \psi' -  \frac{g^2 \rho_0}{P'(\rho_0)} \psi^2=  -g  \jump{\rho_0} \psi^2(0).
\end{equation}
Combining these equalities, we can rewrite the energy as 
\begin{equation}\label{energy_div_form}
 E(\varphi,\psi) = -\frac{g}{2}  \jump{\rho_0}\psi^2(0) + \hal \int_{-m}^\ell  \left(\sqrt{P'(\rho_0)\rho_0}(\psi'+\abs{\xi} \varphi) - \frac{g\sqrt{\rho_0}}{\sqrt{P'(\rho_0)}}  \psi  \right)^2.
\end{equation}
From the non-negativity of the integrals, we deduce that if $E(\varphi,\psi)<0$, then $\psi(0) \neq 0$.
\end{proof}

With Lemma  \ref{neg_inf} in hand, we way apply the direct methods to deduce the existence of a minimizer of $E$ in $\mathcal{A}$.

\begin{prop}\label{min_exist}
$E$ achieves its infimum on $\mathcal{A}$.
\end{prop}
\begin{proof}
First note that by employing the identity $-2ab=(a-b)^2-(a^2+b^2)$ and the constraint  $J(\varphi,\psi)=1$ we may rewrite 
\begin{equation}\label{E_lower_bound}
E(\varphi,\psi) = -g \abs{\xi} +  \hal \int_{-m}^\ell   P'(\rho_0)\rho_0 (\psi' + \abs{\xi} \varphi)^2    
+  g\abs{\xi} \rho_0 (\varphi-\psi)^2 \ge -g\abs{\xi}.
\end{equation}
This shows that $E$ is bounded below on $\mathcal{A}$.  Let $(\varphi_n,\psi_n)\in\mathcal{A}$ be a minimizing sequence.  Then $\varphi_n$ is bounded in $L^2((-m,\ell))$ and $\psi_n$ is bounded in $H_0^1((-m,\ell))$, so up to the extraction of a subsequence   $\varphi_n \rightharpoonup \varphi$ weakly in $L^2$, $\psi_n \rightharpoonup \psi$ weakly in $H_0^1$, and $\psi_n \rightarrow \psi$ strongly in $L^2$.  Via weak lower semi-continuity and the strong $L^2$ convergence $\psi_n \rightarrow \psi$, we then have that
\begin{equation}
 E(\varphi,\psi) \le \liminf_{n\rightarrow \infty} E(\varphi_n,\psi_n) = \inf_{\mathcal{A}} E.
\end{equation}
All that remains is to show that $(\varphi,\psi)\in\mathcal{A}$.

Again by lower semicontinuity, we know that $J(\varphi,\psi) \le 1$.  Suppose by way of contradiction that $J(\varphi,\psi) <1$.  By the homogeneity of $J$ we may find $\alpha>1$ so that $J(\alpha \varphi, \alpha \psi) =1$, i.e. we may scale up $(\varphi,\psi)$ so that $(\alpha \varphi, \alpha \psi) \in \mathcal{A}$.  By Lemma \ref{neg_inf} we know that $\inf_{\mathcal{A}} E<0$, and from this we deduce that 
\begin{equation}
 E(\alpha \varphi, \alpha \psi) = \alpha^2 E(\varphi,\psi) \le \alpha^2 \inf_{\mathcal{A}} E < \inf_{\mathcal{A}} E,
\end{equation}
which is a contradiction since $(\alpha \varphi,\alpha \psi) \in \mathcal{A}$.  Hence $J(\varphi,\psi)=1$ so that $(\varphi,\psi)\in\mathcal{A}$.
\end{proof}

\begin{remark}
 By inequality \eqref{E_lower_bound} we know that $-\lambda^2 = E(\varphi,\psi) \ge - g\abs{\xi}$ and hence $\lambda \le \sqrt{g \abs{\xi}}$. 
\end{remark}

The next result provides a lower bound for $\lambda$ for large $\abs{\xi}$, showing that $\lambda(\abs{\xi}) \rightarrow \infty$ as $\abs{\xi} \rightarrow \infty$.

\begin{lem}\label{eigen_lower}
There exist constants $C_0, C_1, C_2>0$ depending on $\rho^\pm_0, P_{\pm}, g, m, \ell$ such that the eigenvalue $\lambda = \lambda(\abs{\xi})$ satisfies
\begin{equation}
 \lambda^2 \ge C_1 \abs{\xi} - C_2 \text{ for } \abs{\xi} \ge C_0.
\end{equation}
\end{lem}
\begin{proof}
For $\alpha \ge 2$ we consider the function $\psi_\alpha$ defined by 
\begin{equation}
 \psi_\alpha(x_3) = 
\begin{cases}
 \left(1-\frac{x_3}{\ell} \right)^{\alpha/2}, & x_3 \in [0,\ell) \\
  \left(1+\frac{x_3}{m} \right)^{\alpha/2}, &  x_3\in (-m,0). 
\end{cases}
\end{equation}
Then  according to \eqref{n_i_1} and \eqref{n_i_2}
\begin{equation}
 E(-\psi_\alpha'/\abs{\xi},\psi_\alpha) \le -A_1 + \frac{A_2}{1+\alpha}
\end{equation}
for two constants $A_1,A_2 >0$ depending on $\rho^\pm_0, P_{\pm}, g, m, \ell$ but not on $\alpha $ or $\abs{\xi}$.  On the other hand, since $\rho_0$ is bounded below and above, straightforward calculations yield the bounds
\begin{equation}
 J(-\psi_\alpha'/\abs{\xi},\psi_\alpha) = \hal \int_{-m}^\ell \rho_0 \left( \psi_\alpha^2 + \frac{(\psi_\alpha')^2}{\abs{\xi}^2}\right) \le \frac{A_3}{1+\alpha} + \frac{1}{\abs{\xi}^2}\left( \frac{A_4}{1+\alpha} + A_5 + \alpha A_6 \right) 
\end{equation}
and 
\begin{equation}
 J(-\psi_\alpha'/\abs{\xi},\psi_\alpha) \ge \frac{A_7}{1+\alpha}
\end{equation}
for constants $A_3, A_4, A_5, A_6, A_7>0$ depending on the same parameters but not $\alpha$ or $\abs{\xi}$.  From this we see that for $\alpha = \abs{\xi} \ge C_0$ sufficiently large, there are positive constants $C_1, C_2$ so that
\begin{equation}
 \frac{E(-\psi_\alpha'/\abs{\xi},\psi_\alpha)}{J(-\psi_\alpha'/\abs{\xi},\psi_\alpha)} \le  -C_1 \abs{\xi} + C_2.
\end{equation}
Since $-\lambda^2 = \inf E/J$, the result follows.
\end{proof}

We now show that the minimizers of Proposition \ref{min_exist} satisfy Euler-Langrange equations equivalent to \eqref{coupled}.

\begin{prop}\label{e_l_eqns}
Let $(\varphi,\psi)\in \mathcal{A}$ be the minimizers of $E$ constructed in Proposition \ref{min_exist}.  Let $\mu := E(\varphi,\psi) <0$.  Then $(\varphi,\psi)$ satisfy
\begin{equation}\label{e_l_0}
\begin{cases}
 \mu \rho_0 \varphi  =  \abs{\xi} ( P'(\rho_0)\rho_0  (\psi' + \abs{\xi} \varphi)) -  g\abs{\xi}\rho_0 \psi  \\
 \mu \rho_0 \psi  = - ( P'(\rho_0)\rho_0 (\psi' + \abs{\xi} \varphi )) ' - g \abs{\xi} \rho_0 \varphi 
\end{cases}
\end{equation}
along with the jump conditions
\begin{equation}
 \jump{\psi}=0 \text{ and }\jump{ P'(\rho_0)\rho_0 (\abs{\xi} \varphi + \psi') }     = 0
\end{equation}
and the boundary conditions $\psi(-m) = \psi(\ell) =0$.  Moreover, the solutions are smooth when restricted to either $(-m,0)$ or $(0,\ell)$.
\end{prop}

\begin{proof}
 Fix $(\varphi_0,\psi_0)\in L^2((-m,\ell)) \times H_0^1((-m,\ell))$.   Define
\begin{equation}
j(t,s)= J(\varphi+t\varphi_0 + s \varphi,\psi+t\psi_0 + s \psi) 
\end{equation}
and note that $ j(0,0) = 1$.  Moreover $j$ is smooth and
\begin{equation}
 \frac{\partial j}{\partial t}(0,0) = \int_{-m}^\ell   \rho_0  \left[ \varphi_0 \varphi  + \psi_0 \psi  \right] 
\end{equation}
and
\begin{equation}
 \frac{\partial j}{\partial s}(0,0) = \int_{-m}^\ell   \rho_0  (\varphi^2 + \psi^2)  =2.
\end{equation}
So, by the inverse function theorem, we can solve for $s = \sigma(t)$ in a neighborhood of $0$ as a $C^1$ function so that $\sigma(0)=0$ and $j(t,\sigma(t))=1$.  We may differentiate the last equation to find
\begin{equation}
 \frac{\partial j}{\partial t}(0,0) + \frac{\partial j}{\partial s}(0,0) \sigma'(0) = 0,
\end{equation}
and hence that
\begin{equation}
 \sigma'(0) = -\hal \frac{\partial j}{\partial t}(0,0) = -\hal \int_{-m}^\ell   \rho_0  \left[ \varphi_0 \varphi  + \psi_0 \psi  \right].
\end{equation}

Since $(\varphi,\psi)$ are minimizers over $\mathcal{A}$,  we then have
\begin{equation}
 0 = \left. \frac{d}{dt}\right\vert_{t=0} E(\varphi+t\varphi_0 + \sigma(t) \varphi,\psi+t\psi_0 + \sigma(t) \psi),
\end{equation}
which implies that 
\begin{multline}
0= \int_{-m}^\ell P'(\rho_0)\rho_0 (\psi'+\abs{\xi} \varphi)(\psi_0' + \sigma'(0) \psi' + \abs{\xi} \varphi_0 + \abs{\xi} \sigma'(0) \varphi) \\
- g \abs{\xi} \rho_0 (\psi(\varphi_0 + \sigma'(0) \varphi) + \varphi(\psi_0 + \sigma'(0)\psi)    ).
\end{multline}
Rearranging and plugging in the value of $\sigma'(0)$, we may rewrite this equation as
\begin{equation}\label{eigenvalue_form}
 \int_{-m}^\ell  P'(\rho_0)\rho_0 (\psi'+\abs{\xi} \varphi) ( \psi_0' + \abs{\xi} \varphi_0) - g \abs{\xi} \rho_0(\psi \varphi_0 + \varphi \psi_0)
= \mu \int_{-m}^\ell   \rho_0  \left[ \varphi_0 \varphi  + \psi_0 \psi  \right]
\end{equation}
where the Lagrange multiplier (eigenvalue) is $ \mu = E(\varphi,\psi).$

By making variations with $\varphi_0, \psi_0$ compactly supported in either $(-m,0)$ or $(0,\ell)$, we find that $\varphi$ and $\psi$ satisfy the equations \eqref{e_l_0} in a weak sense in $(-m,0)$ and $(0,\ell)$.  Standard bootstrapping arguments then show that $(\varphi,\psi)$ are in $H^k((-m,0))$ (resp. $H^k((0,\ell))$)  for all $k\ge 0$ when restricted to $(-m,0)$ (resp. $(0,\ell)$), and hence the functions are smooth when restricted to either interval.  This implies that the equations are also classically satisfied on $(-m,0)$ and $(0,\ell)$.  Since $(\varphi,\psi)\in H^2$, the traces of the functions and their derivatives are well-defined at the endpoints $x_3 = -m,0,\ell$, and it remains to show that the jump conditions are satisfied at $x_3 =0$ and the boundary conditions satisfied at $x_3 =-m,\ell$.  Making variations with respect to arbitrary $\varphi_0$, $\psi_0 \in C_c^\infty((-m,\ell))$, we find that the jump condition
\begin{equation}
\jump{ P'(\rho_0)\rho_0 (\abs{\xi} \varphi + \psi') }     = 0.
\end{equation}
must be satisfied.  Note that the conditions $\jump{\psi}=0$  and $\psi(-m)=\psi(\ell)=0$ are satisfied trivially since $\psi \in H_0^1((-m,\ell)) \hookrightarrow C_0^{0,1/2}((-m,\ell))$.  
\end{proof}

We may now deduce the existence of a growing mode solution.

\begin{thm}\label{coupled_solution}
For any $\abs{\xi}> 0$ there exists a solution $(\varphi,\psi)$ with $\lambda = \lambda(\abs{\xi}) > 0$ to \eqref{coupled} satisfying the appropriate jump and boundary conditions.  For these solutions $\psi(0)\neq 0$, and the solutions are smooth when restricted to $(-m,0)$ or $(0,\ell)$.
\end{thm}
\begin{proof}
Let $\lambda^2(\xi) = - \mu = - E(\varphi,\psi) > 0$, where $(\varphi,\psi)$ are the minimizers of $E$ over the set $\mathcal{A}$ constructed in Proposition \ref{min_exist}.  Then Proposition \ref{e_l_eqns} shows that $(\varphi,\psi)$ solve the equations \eqref{coupled} with $\lambda = \sqrt{-\mu} >0$.  Lemma \ref{non_zero_origin} implies that $\psi(0)\neq 0$.
\end{proof}

The next result provides an estimate for the $H^k$ norm of the solutions $(\varphi,\psi)$ with $\abs{\xi}$ varying, which will be useful in the next section when such solutions are integrated in a Fourier synthesis.  To emphasize the dependence on $\abs{\xi}$ we will write these solutions as $\varphi(\abs{\xi})  = \varphi(\abs{\xi},x_3)$ and $\psi(\abs{\xi})=\psi(\abs{\xi},x_3)$.

\begin{lem}\label{sobolev_bounds}
Let $\varphi(\abs{\xi}),\psi(\abs{\xi})$ be the solutions to \eqref{coupled} constructed in Theorem \ref{coupled_solution}.  Let $C_0,C_1,C_2>0$ be the constants from  Lemma \ref{eigen_lower}.  Fix $0<R_1<\infty$ so that 
\begin{equation}
R_1 > \max\{C_0,2C_2/C_1\}. 
\end{equation}
Then for $\abs{\xi} \ge R_1$, for each $k \ge 0$ there exists a constant $A_k >0$ depending on $\rho^\pm_0, P_{\pm}, g, m, \ell$ so that
\begin{equation}\label{s_b_0}
 \norm{\varphi(\abs{\xi})}_{H^k((-m,0))} + \norm{\psi(\abs{\xi})}_{H^k((-m,0))}  + \norm{\varphi(\abs{\xi})}_{H^k((0,\ell))} + \norm{\psi(\abs{\xi})}_{H^k((0,\ell))} \le A_k \sum_{j=0}^k \abs{\xi}^j.
\end{equation}
Also, there exists a $B_0>0$ depending on the same parameters so that for any $\abs{\xi}>0$
\begin{equation}\label{s_b_00}
 \norm{\sqrt{\varphi^2(\abs{\xi}) + \psi^2(\abs{\xi})}}_{L^2((-m,\ell))} \ge B_0.
\end{equation}

\end{lem}
\begin{proof}
We begin with the proof of \eqref{s_b_0}.  For simplicity we will prove an estimate of the $H^k$ norms only on the interval $(0,\ell)$.  A bound on $(-m,0)$ follows similarly, and the result follows by adding the two.  First note that the choice of $R_1$, when combined with Lemma \ref{eigen_lower} and the bound \eqref{E_lower_bound} implies that 
\begin{equation}\label{s_b_1}
 -g \abs{\xi} \le \mu \le C_2 - C_1\abs{\xi} \le -\frac{C_1}{2} \abs{\xi}.
\end{equation}
Recall also that $\rho_0$ is smooth on each interval $(0,\ell)$ and $(-m,0)$ and bounded above and below.  Throughout the proof we will let $C>0$ denote a generic constant depending on the appropriate parameters.

We proceed by induction on $k$.  For $k=0$ the fact that $(\varphi(\abs{\xi}),\psi(\abs{\xi}))\in \mathcal{A}$ implies that there is a constant $A_0>0$ depending on the various parameters so that
\begin{equation}
\norm{\varphi(\abs{\xi})}_{L^2((0,\ell))} + \norm{\psi(\abs{\xi})}_{L^2((0,\ell))} \le A_0.
\end{equation}

Suppose now that the bound holds some $k\ge 0$, i.e.
\begin{equation}
 \norm{\varphi(\abs{\xi})}_{H^k((0,\ell))} + \norm{\psi(\abs{\xi})}_{H^k((0,\ell))} \le A_k \sum_{j=0}^k \abs{\xi}^j.
\end{equation}
Define $\omega(\abs{\xi}):= P_+'(\rho_0)\rho_0(\psi'(\abs{\xi}) + \abs{\xi} \varphi(\abs{\xi}))$, where $' = \partial_{x_3}$.  Then \eqref{e_l_0} implies that
\begin{equation}\label{s_b_2}
 \begin{cases}
  \omega(\abs{\xi}) = \mu \abs{\xi}^{-1} \rho_0 \varphi(\abs{\xi}) + g \rho_0 \psi(\abs{\xi}) \\
  \omega'(\abs{\xi}) = - \mu \rho_0 \psi(\abs{\xi}) - g \abs{\xi} \rho_0 \varphi(\abs{\xi}).
 \end{cases}
\end{equation}
These equations and \eqref{s_b_1} then imply that 
\begin{equation}
 \norm{\omega(\abs{\xi})}_{H^k((0,\ell))} \le C( \norm{\varphi(\abs{\xi})}_{H^k((0,\ell))} + \norm{\psi(\abs{\xi})}_{H^k((0,\ell))})
\end{equation}
and
\begin{equation}
 \norm{\omega'(\abs{\xi})}_{H^k((0,\ell))} \le C \abs{\xi} ( \norm{\varphi(\abs{\xi})}_{H^k((0,\ell))} + \norm{\psi(\abs{\xi})}_{H^k((0,\ell))})
\end{equation}
so that $\norm{\omega(\abs{\xi})}_{H^{k+1}((0,\ell))} \le C \sum_{j=0}^{k+1} \abs{\xi}^j$.  But then the definition of $\omega(\abs{\xi})$ implies that 
\begin{equation}
 \norm{\psi'(\abs{\xi})}_{H^k((0,\ell))} \le   \norm{\frac{\omega(\abs{\xi})}{P'_+(\rho_0)\rho_0} }_{H^k((0,\ell))} + \abs{\xi} \norm{\varphi(\abs{\xi})}_{H^k((0,\ell))}
\end{equation}
so that $\norm{\psi(\abs{\xi})}_{H^{k+1}((0,\ell))} \le C \sum_{j=0}^{k+1} \abs{\xi}^j$.  Returning to the first equation in \eqref{s_b_2}, we see that
\begin{equation}
 \norm{\varphi(\abs{\xi})}_{H^{k+1}((0,\ell))} \le \frac{g\abs{\xi}}{\mu} \norm{ \psi(\abs{\xi})}_{H^{k+1}((0,\ell))}  + \frac{\abs{\xi}}{\mu} \norm{\frac{\varphi(\abs{\xi})}{\rho_0}  }_{H^{k+1}((0,\ell))}.
\end{equation}
Invoking the bound \eqref{s_b_1} again, we deduce that $\norm{\varphi(\abs{\xi})}_{H^{k+1}((0,\ell))} \le C \sum_{j=0}^{k+1} \abs{\xi}^j$.  Hence 
\begin{equation}
 \norm{\varphi(\abs{\xi})}_{H^{k+1}((0,\ell))} + \norm{\psi(\abs{\xi})}_{H^{k+1}((0,\ell))} \le A_{k+1} \sum_{j=0}^{k+1} \abs{\xi}^j
\end{equation}
for some constant $A_{k+1}>0$ depending on the parameters, i.e. the bound holds for $k+1$.  By induction, the bound holds for all $k\ge 0$.

To prove \eqref{s_b_00} we again utilize the fact that for any $\abs{\xi} >0$ we have  $(\varphi(\abs{\xi}),\psi(\abs{\xi}))\in \mathcal{A}$.  Since $\rho_0$ is bounded above and below, the bound follows.

\end{proof}

A solution to \eqref{coupled} gives rise to a solution to the system of equations \eqref{w_1_equation}--\eqref{w_3_equation} for the growing mode velocity, $w$, as well.

\begin{cor}\label{w_soln}
For any $\xi =(\xi_1,\xi_2) \neq (0,0)$ there exists a solution $\varphi = \varphi(\xi,x_3)$, $\theta = \theta(\xi,x_3)$, $\psi = \psi(\abs{\xi},x_3)$  and $\lambda(\abs{\xi})>0$ to  \eqref{w_1_equation}--\eqref{w_3_equation} satisfying the appropriate jump and boundary conditions so that $\psi(0)\neq 0$; the solutions are smooth when restricted to $(-m,0)$ or $(0,\ell)$.  The solutions are equivariant in $\xi$ in the sense that if $R\in SO(2)$ is a rotation operator, then 
\begin{equation}
\begin{pmatrix}
\varphi(R \xi,x_3) \\ \theta(R \xi,x_3) \\ \psi(R \xi,x_3) 
\end{pmatrix}
= 
\begin{pmatrix}
R_{11} & R_{12} & 0 \\ 
R_{21} & R_{22} & 0 \\ 
0      & 0      & 1
\end{pmatrix}
\begin{pmatrix}
\varphi(\xi,x_3) \\ \theta(\xi,x_3) \\ \psi(\xi,x_3) 
\end{pmatrix}.
\end{equation}
\end{cor}
\begin{proof}
 We may find a rotation operator $R \in SO(2)$ so that $R \xi = (\abs{\xi},0)$.  Define
\begin{equation}
(\varphi(\xi,x_3),\theta(\xi,x_3)) = R^{-1} (\varphi(\abs{\xi},x_3),0) 
\end{equation}
and $\psi(\xi,x_3) = \psi(\abs{\xi},x_3)$, where $\varphi(\abs{\xi})$ and $\psi(\abs{\xi})$ are the solutions from Theorem \ref{coupled_solution}.  This gives a solution to \eqref{w_1_equation}--\eqref{w_3_equation}.  The equivariance in $\xi$ follows from the definition.
\end{proof}

\subsection{Fourier synthesis}

In this section we will use Fourier synthesis to build growing solutions to \eqref{linearized} out of the solutions constructed in the previous section (Corollary \ref{w_soln}) for fixed spatial frequency $\xi \in \Rn{2}$.   The solutions will be constructed to grow in the piecewise Sobolev space of order $k$, $H^k$, defined by \eqref{sob_def}.

\begin{thm}\label{growing_mode_soln}
Let $R_1\le R_2 < R_3 < \infty,$ where $R_1>0$ is the constant from Lemma \ref{sobolev_bounds}.  Let $f\in C_c^\infty(\Rn{2})$ be a real-valued function so that $f(\xi) = f(\abs{\xi})$ and $\supp(f)\subset B(0,R_3)\backslash B(0,R_2)$.  For $\xi \in \Rn{2}$ define
\begin{equation}
 \hat{w}(\xi,x_3) =  -i \varphi(\xi,x_3) e_1 - i \theta(\xi,x_3) e_2 + \psi(\xi,x_3) e_3,
\end{equation}
where $\varphi,\theta,\psi$ are the solutions provided by Corollary \ref{w_soln}.  Writing $x'\cdot \xi = x_1 \xi_1 + x_2 \xi_2$, we define
\begin{equation}\label{g_m_s_1}
 \eta(x,t) = \frac{1}{4\pi^2} \int_{\Rn{2}} f(\xi) \hat{w}(\xi,x_3) e^{\lambda(\abs{\xi}) t} e^{i x'\cdot \xi} d\xi,
\end{equation}
\begin{equation}
 v(x,t) = \frac{1}{4\pi^2}  \int_{\Rn{2}}\lambda(\abs{\xi}) f(\xi) \hat{w}(\xi,x_3) e^{\lambda(\abs{\xi}) t} e^{i x'\cdot \xi} d\xi ,
\end{equation}
and
\begin{equation}\label{g_m_s_2}
 q(x,t) 
= -\frac{\rho_0(x_3)}{4\pi^2}  \int_{\Rn{2}} f(\xi) (\xi_1 \varphi(\xi,x_3) +  \xi_2 \theta(\xi,x_3) + \partial_{x_3} \psi(\xi,x_3)) e^{\lambda(\abs{\xi}) t} e^{i x'\cdot \xi} d\xi.
\end{equation}
Then $\eta,v,q$ are real-valued solutions to the linearized equations \eqref{linearized} along with the corresponding jump and boundary conditions.  The solutions are equivariant in the sense that if $R\in SO(3)$ is a rotation that keeps the vector $e_3$ fixed, then
\begin{equation}\label{g_m_s_00}
\eta(Rx,t) = R \eta(x,t), v(Rx,t) = R v(x,t), \text{and } q(Rx,t) = q(x,t).
\end{equation}
For every $k \in \mathbb{N}$ we have the estimate
\begin{equation}\label{g_m_s_0} 
\norm{\eta(0)}_{H^k} + \norm{v(0)}_{H^k} + \norm{q(0)}_{H^k} \le \bar{C}_k \left( \int_{\Rn{2}} (1+\abs{\xi}^2)^{k+1} \abs{f(\xi)}^2 d\xi \right)^{1/2} < \infty
\end{equation}
for a constant $\bar{C}_k>0$ depending on the parameters $\rho^\pm_0, P_{\pm}, g, m, \ell$;  moreover, for every $t > 0$ we have  $\eta(t),v(t),q(t) \in H^k(\Omega)$ and  
\begin{equation}\label{g_m_s_3}
\begin{cases}
e^{t \sqrt{C_1 R_2 - C_2}} \norm{\eta(0)}_{H^k} \le   \norm{\eta(t)}_{H^k} \le e^{t \sqrt{g R_3}} \norm{\eta(0)}_{H^k}     \\
e^{t \sqrt{C_1 R_2 - C_2}} \norm{v(0)}_{H^k} \le \norm{v(t)}_{H^k}  \le  e^{t \sqrt{g R_3}} \norm{v(0)}_{H^k}  \\
e^{t \sqrt{C_1 R_2 - C_2}} \norm{q(0)}_{H^k} \le \norm{q(t)}_{H^k}  \le e^{t \sqrt{g R_3}} \norm{q(0)}_{H^k} 
\end{cases}
\end{equation}
where $C_1,C_2>0$ are the constants from Lemma \ref{eigen_lower} and $\sqrt{C_1 R_2 - C_2}>0$. 
\end{thm}

\begin{proof}
For each fixed $\xi\in \Rn{2}$, 
\begin{equation}
 \eta(x,t) =  f(\xi) \hat{w}(\xi,x_3) e^{\lambda(\abs{\xi}) t} e^{i x'\cdot \xi},
\end{equation}
\begin{equation}
 v(x,t) = \lambda(\abs{\xi}) f(\xi) \hat{w}(\xi,x_3) e^{\lambda(\abs{\xi}) t} e^{i x'\cdot \xi},
\end{equation}
and
\begin{equation}
 q(x,t) 
= -\rho_0(x_3) f(\xi) (\xi_1 \varphi(\xi,x_3) +  \xi_2 \theta(\xi,x_3) + \partial_3 \psi(\xi,x_3)) e^{\lambda(\abs{\xi}) t} e^{i x'\cdot \xi}
\end{equation}
gives a solution to \eqref{linearized}.  Since  $\supp(f) \subset B(0,R_3) \backslash B(0,R_2)$, Lemma \ref{sobolev_bounds} implies that 
\begin{equation}
\sup_{\xi \in \supp(f)} \norm{ \partial_{x_3}^k \hat{w}(\xi,\cdot)}_{L^\infty} < \infty \text{ for all } k\in \mathbb{N}.
\end{equation}
Also, $\lambda(\abs{\xi}) \le \sqrt{g \abs{\xi}}$.  These bounds imply that the Fourier synthesis of these solutions given by \eqref{g_m_s_1}--\eqref{g_m_s_2} is also a solution to \eqref{linearized}.  The Fourier synthesis is real-valued because $f$ is real-valued and radial and because of the equivariance in $\xi$ given in Corollary \ref{w_soln}.  This equivariance in $\xi$ also implies the equivariance of $\eta, v, q$ written in \eqref{g_m_s_00}.    

The bound \eqref{g_m_s_0} follows by applying Lemma \ref{sobolev_bounds}  with arbitrary $k \ge 0$ and utilizing the fact that $f$ is compactly supported.  According to \eqref{s_b_1}
\begin{equation}
\sqrt{g R_3} \ge \sqrt{g \abs{\xi}}\ge  \lambda(\abs{\xi}) \ge \sqrt{C_1 \abs{\xi} - C_2} \ge \sqrt{C_1 R_2 - C_2}  >0,
\end{equation}
which then yields the bounds \eqref{g_m_s_3}.
\end{proof}

\begin{remark}
It holds that
\begin{equation}
 \eta_3(x_1,x_2,0,0) = \frac{1}{4\pi^2}  \int_{\Rn{2}} f(\xi)  \psi(\xi,0) e^{i x'\cdot \xi} d\xi 
\end{equation}
is the vertical component of the initial linearized flow map at the interface between the two fluids.  Since $\psi(\xi,0)\neq 0$ for any choice of $\xi$, a nonzero $f$ in general gives rise to a nonzero $\eta_3(x_1,x_2,0,0)$.

\end{remark}

\section{Analysis of the linear problem}

\subsection{Estimates for band-limited solutions}

Suppose that $\eta,v,q$ are real-valued solutions to \eqref{linearized} along with the corresponding jump and boundary conditions (of course, by linearity, we may also handle complex solutions by taking the real and complex parts and proceeding with an analysis of each part).  Further suppose that the solutions are band-limited at radius $R>0$, i.e. that
\begin{equation}
 \bigcup_{x_3 \in (-m,\ell)} \supp( \abs{\hat{\eta}(\cdot,x_3)} +\abs{\hat{v}(\cdot,x_3))} + \abs{\hat{q}(\cdot,x_3)}) \subset B(0,R)
\end{equation}
where $\hat{v}$ denotes the horizontal Fourier transform defined by \eqref{hft}.  The importance of the band-limited assumption lies in the fact that the growing mode exponent remains bounded for $\abs{\xi} \le R$.  We henceforth denote 
\begin{equation}\label{max_def}
\Lambda(R):= \sup_{0\le \abs{\xi} \le R} \lambda(\abs{\xi})  \le \sqrt{g R} < \infty. 
\end{equation}
We will derive estimates for band-limited solutions in terms of $\Lambda(R)$.

It will be convenient to work with a second-order formulation of the equations.  To arrive at this, we differentiate the third equation in \eqref{linearized} in time and eliminate the $q$ and $\eta$ terms using the first and second equations.  This yields the equation
\begin{equation}\label{second_order}
 \rho_0 \partial_{tt} v - \nab(P'(\rho_0)\rho_0 \diverge{v}) +  g \rho_0 \nab v_3 -  g \rho_0 \diverge{v} e_3 =0.
\end{equation}
coupled to the jump conditions
\begin{equation}
\jump{\dt v_3} =0 \text{ and } \jump{ P'(\rho_0) \rho_0 \diverge{v}   } =  0
\end{equation}
and the boundary condition $\dt v_3(x_1,x_2,-m,t) = \dt v_3(x_1,x_2,\ell,t)=0$.  The band limited assumption implies that $\supp(\hat{v}(\cdot,x_3)) \subset B(0,R)$ for all $x_3\in (-m,\ell)$.  The initial data for $\dt v(0)$ is given in terms of the initial data $q(0)$ and $\eta(0)$ via the third linear equation, i.e. $\dt v(0)$ satisfies
\begin{equation}
  \rho_0 \dt v(0) = -\nab(P'(\rho_0) q(0))  - g q(0) e_3 - g \rho_0 \nab \eta_3(0).
\end{equation}

Our first result gives an evolution equation for an energy associated to such solutions.

\begin{lem}\label{lin_en_evolve}
For solutions to \eqref{second_order} it holds that
\begin{equation}
 \dt \int_{\Omega} \left( \frac{\rho_0}{2} \abs{\dt v}^2 + \frac{P'(\rho_0)\rho_0}{2} \abs{\diverge{v} - \frac{g}{P'(\rho_0)} v_3 }^2 \right)
= \dt \int_{\Rn{2}} \frac{g \jump{\rho_0}}{2} \abs{v_3}^2. 
\end{equation}
\end{lem}
\begin{proof}
Recall that $\Omega_+ = \Rn{2} \times(0,\ell)$.  Take the dot product of \eqref{second_order} with $\dt v(t)$ and integrate over $\Omega_+$.  After integrating by parts and utilizing \eqref{enthalpy_eqn}, we get
\begin{multline}
 \int_{\Omega_+}  \rho_0 \dt v \cdot \partial_{tt}v + P'(\rho_0)\rho_0 (\diverge{v}) (\diverge{\dt v}) -g\rho_0(v_3 \diverge{\dt v} + \dt v_3 \diverge{v}) + \frac{g^2 \rho_0}{P'(\rho_0)} v_3 \dt v_3
\\
= \int_{\Rn{2}} g \rho^+_0  v_3 \dt v_3 
- \int_{\Rn{2}}  P_+'(\rho^+_0)\rho^+_0  \diverge{v} \dt v_3.
\end{multline}
We may then pull the time derivatives outside the first and second integrals to arrive at the equality 
\begin{multline}
 \dt \int_{\Omega_+} \left( \rho_0 \frac{\abs{\dt v}^2}{2} + \frac{P'(\rho_0)\rho_0}{2} \abs{\diverge{v} - \frac{g}{P'(\rho_0)} v_3 }^2 \right)
\\
= \dt \int_{\Rn{2}} g \rho^+_0 \frac{\abs{v_3}^2}{2} 
- \int_{\Rn{2}}  P_+'(\rho^+_0)\rho^+_0  \diverge{v} \dt v_3.
\end{multline}
A similar result holds on $\Omega_- = \Rn{2} \times(-m,0)$ with the opposite sign on the right hand side.  Adding the two together yields
\begin{multline}
 \dt \int_{\Omega} \rho_0 \frac{\abs{\dt v}^2}{2} + \frac{P'(\rho_0)\rho_0}{2} \abs{\diverge{v} - \frac{g}{P'(\rho_0)} v_3 }^2
\\
= \dt \int_{\Rn{2}} g \jump{\rho_0} \frac{\abs{v_3}^2}{2} 
- \int_{\Rn{2}}  \jump{ P'(\rho_0)\rho_0  \diverge{v} \dt v_3 }.
\end{multline}
The jump conditions imply that the second term on the right hand side of this equation vanishes, and the result follows.
\end{proof}

The next result allows us to estimate the energy in terms of $\Lambda(R)$.

\begin{lem}\label{lin_en_bound}
Let $v\in H^1(\Omega)$ be band-limited at radius $R>0$ and satisfy the boundary conditions  $v_3(x_1,x_2,-m,t)=v_3(x_1,x_2,\ell,t)=0$.  Then 
\begin{equation}
\int_{\Rn{2}} \frac{g \jump{\rho_0}}{2} \abs{v_3}^2  -  \int_\Omega \frac{P'(\rho_0)\rho_0 }{2}\abs{ \diverge{v} - \frac{g}{P'(\rho_0)} v_3 }^2 \\
\le \frac{\Lambda^2(R)}{2}\int_\Omega \rho_0 \abs{v}^2, 
\end{equation}
where $\Lambda(R)$ is defined by \eqref{max_def}.
 \end{lem}
\begin{proof}
 Take the horizontal Fourier transform \eqref{hft} and apply \eqref{parseval} to see that 
\begin{multline}
4\pi^2\int_{\Rn{2}} \frac{g \jump{\rho_0}}{2} \abs{v_3}^2  -  4\pi^2\int_\Omega \frac{P'(\rho_0)\rho_0 }{2}\abs{ \diverge{v} - \frac{g}{P'(\rho_0) }v_3 }^2 \\
= \int_{\Rn{2}} \frac{g \jump{\rho_0}}{2} \abs{\hat{v}_3}^2 d\xi -  \int_\Omega \frac{P'(\rho_0)\rho_0}{2}\abs{ i\xi_1 \hat{v}_1 + i\xi_2 \hat{v}_2 + \partial_{x_3} \hat{v}_3 - \frac{g}{P'(\rho_0)} \hat{v}_3 }^2 d\xi dx_3\\
=  \int_{\Rn{2}} \left( \frac{g \jump{\rho_0}}{2} \abs{\hat{v}_3}^2- \int_{-m}^\ell \frac{P'(\rho_0)\rho_0}{2}\abs{ i\xi_1 \hat{v}_1 + i\xi_2 \hat{v}_2 + \partial_{x_3} \hat{v}_3 - \frac{g}{P'(\rho_0)} \hat{v}_3 }^2 dx_3 \right) d\xi .
\end{multline}
Consider now the last integrand for fixed $(\xi_1,\xi_2)= \xi \neq 0$, writing $\varphi(x_3)= i \hat{v}_1(\xi_1,\xi_2,x_3)$, $\theta(x_3)= i \hat{v}_2(\xi_1,\xi_2,x_3)$, $\psi(x_3) = \hat{v}_3(\xi_1,\xi_2,x_3)$.  That is, define
\begin{equation}
Z(\varphi,\theta,\psi;\xi) = \frac{g \jump{\rho_0}}{2} \abs{\psi(0)}^2 - \int_{-m}^\ell \frac{P'(\rho_0)\rho_0}{2}\abs{ \xi_1 \varphi + \xi_2 \theta + \psi' - \frac{g}{P'(\rho_0)} \psi }^2 dx_3
\end{equation}
where $' = \partial_{x_3}$.  By splitting
\begin{equation}
 Z(\varphi,\theta,\psi;\xi) = Z(\Re \varphi,\Re \theta,\Re \psi;\xi) +Z(\Im \varphi,\Im \theta,\Im \psi;\xi) 
\end{equation}
we may reduce to bounding $Z$ when $\varphi,\theta,\psi$ are real-valued functions, and then apply the bound to the real and imaginary parts of $\varphi,\theta,\psi$. 

The expression for $Z$ is invariant under simultaneous rotations of $(\xi_1,\xi_2)$ and $(\varphi,\theta)$, so without loss of generality we may assume that $(\xi_1,\xi_2) = (\abs{\xi},0)$ with $\abs{\xi} > 0$.   Then according to \eqref{energy_div_form}, $Z(\varphi,\theta,\psi;\xi) = -E(\varphi,\psi)$  and hence
\begin{equation}
 Z(\varphi,\theta,\psi;\xi) \le \frac{\lambda^2(\abs{\xi})}{2} \int_{-m}^\ell \rho_0 (\varphi^2 +\psi^2),
\end{equation}
where $\lambda(\abs{\xi})$ is the growing-mode exponent constructed in the previous section.  Translating the inequality back to the original notation for fixed $\xi$, we find
\begin{equation}
 \frac{g \jump{\rho_0}}{2} \abs{\hat{v}_3}^2 - \int_{-m}^\ell \frac{P'(\rho_0)\rho_0}{2}\abs{ i\xi_1 \hat{v}_1 + i\xi_2 \hat{v}_2 + \partial_3 \hat{v}_3 - \frac{g}{P'(\rho_0)}\hat{v}_3 }^2 dx_3 \\
\le \frac{\lambda^2(\abs{\xi})}{2} \int_{-m}^\ell \rho_0 \abs{\hat{v}}^2.
\end{equation}
Integrating each side of this inequality over $\xi \in\Rn{2}$, employing the fact that $\lambda^2(\abs{\xi})\le \Lambda^2(R)$ on the support of $\hat{v}$, and using \eqref{parseval} then proves the result.
\end{proof}

We may now derive growth estimates in terms of the initial data and the value of $\Lambda(R)$.

\begin{prop}\label{lin_growth_bound}
Let $v$ be a solution to \eqref{second_order} along with the corresponding jump and boundary conditions that is also band-limited at radius $R>0$.  Then
\begin{multline}
\norm{ v(t)}_{L^2(\Omega)}^2  +\norm{\dt v(t)}_{L^2(\Omega)}^2 \\
 \le  C e^{2\Lambda(R) t}\left(\norm{v(0)}_{L^2(\Omega)}^2 + \norm{\dt v(0)}_{L^2(\Omega)}^2   + \norm{\diverge{v}(0)}_{L^2(\Omega)}^2 \right)
\end{multline}
for a constant $0<C = C(\rho^\pm_0,P_{\pm},g,\Lambda(R),m,\ell)$.
\end{prop}
\begin{proof}
Integrate the result of Lemma \ref{lin_en_evolve} in time from $0$ to $t$ to find that
\begin{equation}
  \int_\Omega \rho_0 \frac{\abs{\dt v(t)}^2}{2}  \le A + \int_{\Rn{2}} \frac{g \jump{\rho_0}}{2} \abs{v_3(t)}^2 - \int_\Omega \frac{P'(\rho_0)\rho_0}{2}\abs{ \diverge{v(t)} - \frac{g}{P'(\rho_0)} v_3(t) }^2
\end{equation}
where 
\begin{equation}
 A = \int_\Omega \rho_0 \frac{\abs{\dt v(0)}^2}{2} + \int_\Omega \frac{P'(\rho_0)\rho_0}{2}\abs{ \diverge{v(0)} - \frac{g}{P'(\rho_0)} v_3(0) }^2.
\end{equation}
We may then apply Lemma \ref{lin_en_bound} to get the inequality 
\begin{equation}
  \int_\Omega \rho_0 \frac{\abs{\dt v(t)}^2}{2} 
\le A + \frac{\Lambda^2(R)}{2}\int_\Omega \rho_0 \abs{v(t)}^2 .
\end{equation}
Defining the weighted $L^2$ norm by $\norm{f}^2 = \int_\Omega \rho_0 \abs{f}^2,$ we may compactly rewrite the previous inequality as
\begin{equation}\label{l_g_b_1}
 \hal \norm{\dt v(t)}^2   \le A + \frac{\Lambda^2(R)}{2}\norm{ v(t)}^2. 
\end{equation}

An application of Cauchy's inequality shows that 
\begin{equation}
\Lambda(R) \dt \norm{v(t)}^2 = \Lambda(R) 2\langle \dt v(t), v(t) \rangle \le \Lambda^2(R) \norm{v(t)}^2 + \norm{\dt v(t)}^2,
\end{equation}
where $\langle \cdot,\cdot \rangle$ is the weighted $L^2$ inner-product associated with the weighted $L^2$ norm.  We may combine this inequality with \eqref{l_g_b_1} to derive the differential inequality
\begin{equation}
 \dt   \norm{v(t)}^2 \le B + 2\Lambda(R) \norm{v(t)}^2 
\end{equation}
for $B = 2A/\Lambda(R).$  Gronwall's inequality then shows that
\begin{equation}\label{l_g_b_3}
 \norm{v(t)}^2  \le e^{2\Lambda(R) t}\norm{v(0)}^2 + \frac{B}{2\Lambda(R)} (e^{2\Lambda(R) t}-1)
\end{equation}
for all $t\ge 0$.   To derive the corresponding bound for  $\norm{\dt v(t)}^2$ we return to \eqref{l_g_b_1} and plug in \eqref{l_g_b_3} to see that
\begin{equation}
 \frac{1}{2} \norm{\dt v(t)}^2  \le A + \frac{\Lambda^2(R)}{2}\left( e^{2\Lambda(R) t} \norm{v(0)}^2 + \frac{B}{2\Lambda(R)}(e^{2\Lambda(R) t}-1) \right).
\end{equation}
The result follows by noting that the weighted $L^2$ norm $\norm{\cdot}$ is equivalent to the standard $L^2$ norm $\norm{\cdot}_{L^2(\Omega)}$ and that
\begin{equation}
 A \le C \left(\norm{v(0)}_{L^2(\Omega)}^2 + \norm{\dt v(0)}_{L^2(\Omega)}^2   + \norm{\diverge{v(0)}}_{L^2(\Omega)}^2 \right) 
\end{equation}
for a constant $C>0$ depending on the parameters $\rho^\pm_0, P_{\pm}, g,m,\ell.$
\end{proof}

\subsection{Uniqueness}
In order to employ Proposition \ref{lin_growth_bound} in a proof of uniqueness for solutions to the linearized equations, we must first construct a horizontal spatial frequency projection operator.  To this end let $\Phi\in C_c^\infty(\Rn{2})$ be so that $0\le \Phi \le 1$, $\supp(\Phi) \subset B(0,1)$, and $\Phi(x) =1$ for $x\in B(0,1/2)$.  For $R>0$ let $\Phi_R$ be the function defined by $\Phi_R(x) = \Phi(x/R)$.  We define the projection operator $P_R$ via
\begin{equation}
 P_R f = \mathcal{F}^{-1}(\Phi_R \mathcal{F} f ) 
\end{equation}
for $f\in L^2(\Omega)$, where $\mathcal{F}= \hat{}$ denotes the horizontal Fourier transform in $x_1,x_2$ defined by \eqref{hft}.  It is a simple matter to see that $P_R$ satisfies the following.
\begin{enumerate}
 \item $P_R f$ is band-limited at radius $R$.
 \item $P_R$ is a bounded linear operator on $H^k(\Omega)$ for all $k\ge 0$.
 \item $P_R$ commutes with partial differentiation and multiplication by functions depending only on $x_3$.
 \item $P_R f =0$ for all $R>0$ if and and only if $f=0$.
\end{enumerate}

With this operator in hand we may prove the uniqueness result.
\begin{thm}\label{linear_uniqueness}
Solutions to \eqref{linearized} are unique.
\end{thm}
\begin{proof}
It suffices to show that solutions to \eqref{linearized} with $0$ initial data remain $0$ for $t>0$.  Suppose then that $\eta,v,q$ are solutions with vanishing initial data.  Fix $R>0$ and define $\eta_R = P_R \eta$, $v_R = P_R v$, $q_R = P_R q$.  The properties of $P_R$ show that $\eta_R,v_R,q_R$ are also solutions to \eqref{linearized} but that they are band-limited at radius $R$.  Turning to the second order formulation, we find that $v_R$ is a solution to \eqref{second_order} with initial data $v_R(0)=\dt v_R(0) = 0$.  We may then apply Proposition \ref{lin_growth_bound} to deduce that 
\begin{equation}
 \norm{v_R(t)}_{L^2(\Omega)} = \norm{\dt v_R(t)}_{L^2(\Omega)} =0 \text{ for all } t\ge 0,
\end{equation}
which implies that $\eta_R(t), v_R(t),$ and $q_R(t)$ all vanish for $t\ge 0$.  Since $R$ was arbitrary it must hold that $\eta(t)$, $v(t)$ and $q(t)$ also  vanish for $t\ge 0$.
\end{proof}

\subsection{Discontinuous dependence on the initial data}
The solutions to the linear problem \eqref{linearized} constructed in Theorem \ref{growing_mode_soln} are sufficiently pathological to give rise to a result showing that the solutions depend discontinuously on the initial data.  Then, in spite of the previous uniqueness result, we get that the linear problem is ill-posed in the sense of Hadamard.

\begin{thm}\label{linear_ill_posed}
The linear problem \eqref{linearized} with the corresponding jump and boundary conditions is ill-posed in the sense of Hadamard in $H^k(\Omega)$ for every $k$.  More precisely, for any $k,j\in \mathbb{N}$ with $j\ge k$ and for any $T_0>0$ and $\alpha>0$ there exists a sequence of solutions $\{(\eta_n,v_n,q_n)\}_{n=1}^\infty$ to \eqref{linearized}, satisfying the corresponding jump and boundary conditions, so that
\begin{equation}\label{l_i_p_0}
 \norm{\eta_n(0)}_{H^j} + \norm{v_n(0)}_{H^j} + \norm{q_n(0)}_{H^j} \le  \frac{1}{n}
\end{equation}
but
\begin{equation}\label{l_i_p_00}
 \norm{v_n(t)}_{H^k} \ge \norm{\eta_n(t)}_{H^k} \ge \alpha \text{ for all } t\ge T_0.
\end{equation}
\end{thm}

\begin{proof}
Fix $j\ge k\ge 0$, $\alpha>0$, $T_0>0$ and let $\bar{C}_j, R_1, B_0, C_1, C_2 >0$ be the constants from Theorem \ref{growing_mode_soln}, Lemma \ref{sobolev_bounds}, and Lemma \ref{eigen_lower} respectively.  For each $n\in\mathbb{N}$ let $R(n)$ be sufficiently large so that $R(n) > R_1$, $\sqrt{C_1 R(n) - C_2}\ge 1$, and 
\begin{equation}
 \frac{\exp(2 T_0\sqrt{C_1 R(n) - C_2})}{(1+(R(n)+1)^2)^{j-k+1}} \ge \alpha^2 n^2 \frac{\bar{C}_j^2}{B_0^2}.
\end{equation}
Choose $f_n \in C_c^\infty(\Rn{2})$ so that $\supp(f_n) \subset B(0,R(n)+1) \backslash B(0,R(n))$, $f_n$ is real-valued and radial,  and
\begin{equation}\label{l_i_p_1}
 \int_{\Rn{2}} (1+\abs{\xi}^2)^{j+1} \abs{f_n(\xi)}^2 d\xi = \frac{1}{\bar{C}_j^2 n^2}.
\end{equation}

We may now apply Theorem \ref{growing_mode_soln} with $f_n$, $R_2 = R(n)$, and $R_3 = R(n) +1$ to find $\eta_n,v_n,q_n$ that solve \eqref{linearized} and the corresponding jump and boundary conditions  so that $\eta_n(t),v_n(t),q_n(t) \in H^j(\Omega)$ for all $t\ge 0$.  By \eqref{g_m_s_0} and the choice of $f_n$ satisfying \eqref{l_i_p_1}, we have that \eqref{l_i_p_0} holds for all $n$.

We may then estimate
\begin{equation}
\begin{split}
 \norm{\eta_n(T_0)}_{H^k}^2 & \ge \int_{\Rn{2}} (1+\abs{\xi}^2)^k \abs{f_n(\xi)}^2 e^{2 t_0 \lambda(\abs{\xi})} \norm{\hat{w}(\xi,\cdot)}^2_{L^2(-m,\ell)} d\xi \\
&\ge \frac{\exp(2 T_0\sqrt{C_1 R(n) - C_2})}{(1+(R(n)+1)^2)^{j-k+1}}   \int_{\Rn{2}} (1+\abs{\xi}^2)^{j+1} \abs{f_n(\xi)}^2 \norm{\hat{w}(\xi,\cdot)}^2_{L^2(-m,\ell)} d\xi \\
& \ge  \alpha^2 n^2 \frac{\bar{C}_j^2}{B_0^2} \int_{\Rn{2}} (1+\abs{\xi}^2)^{j+1} \abs{f_n(\xi)}^2 B_0^2 d\xi = \alpha^2.
\end{split}
\end{equation}
Here the first bound is trivial, the second  follows since  $\supp(f_n) \subset B(0,R(n)+1)$ and $\lambda(\abs{\xi}) \ge \sqrt{C_1 R(n) -C_2}$, and the third follows from the choice of $R(n)$ and the lower bound \eqref{s_b_00}.  Since $\lambda(\abs{\xi}) \ge \sqrt{C_1 R(n) -C_2} \ge 1$ on the support of $f_n$, we also know that
\begin{equation}
\norm{v_n(t)}_{H^k}^2 \ge \norm{\eta_n(t)}_{H^k}^2 \ge  \norm{\eta_n(T_0)}_{H^k}^2 \text{ for } t \ge T_0,
\end{equation}
from which we deduce \eqref{l_i_p_00}.
\end{proof}

\section{Ill-posedness for the non-linear problem}
Recall that the steady state solution to \eqref{lagrangian_equations} is given by $v=0$, $\eta =\eta^{-1} = Id$, $q=\rho_0$ with $A=I$ and $S_-=S_+ = Id_{\{x_3=0\}}$.  We will now rephrase the non-linear equations \eqref{lagrangian_equations} in a perturbation formulation around the steady state.  Let
\begin{equation}
 \eta = Id + \tilde{\eta},\; \eta^{-1} = Id - \zeta,\; v = 0+ v,\; q = \rho_0 + \sigma,\; A = I - B,
\end{equation}
where 
\begin{equation}
 B^T = \sum_{n=1}^\infty(-1)^{n-1} (D \tilde{\eta})^n.
\end{equation}
In order to deal with the term $h(q) = h(\rho_0 + \sigma)$ we introduce the Taylor expansion
\begin{equation}
 h(\rho_0+\sigma) = h(\rho_0) + h'(\rho_0) \sigma + \mathcal{R}
\end{equation}
where the remainder term is defined by
\begin{equation}
 \mathcal{R}(x,t) = \int_0^{\sigma(x,t)} (\sigma(x,t) - z)h''(\rho_0(x) + z) dz = \int_{\rho_0(x)}^{\rho_0(x) + \sigma(x,t)}(\rho_0(x) + \sigma(x,t) - z) h''(z) dz.
\end{equation}

Then the evolution equations  \eqref{lagrangian_equations}  can be rewritten for $\tilde{\eta}, v, \sigma$ as
\begin{equation}\label{perturb}
 \begin{cases}
\dt \tilde{\eta} = v \\
\dt \sigma + (\rho_0+\sigma)(\diverge{v} - \text{tr}(B Dv)) =0 \\
 \dt v + (I-B)\nab(h'(\rho_0)\sigma + g \tilde{\eta}_3 + \mathcal{R}) = 0. 
 \end{cases}
\end{equation}
We require the compatibility between $\zeta$ and $\tilde{\eta}$ given by
\begin{equation}
 \zeta = \tilde{\eta}\circ(Id-\zeta)
\end{equation}
The jump conditions across the interface are
\begin{equation}\label{pert_jc}
 \begin{cases}
  (v_+(x_1,x_2,0,t) - v_-(S_-(x_1,x_2,t),t) )\cdot n(x_1,x_2,0,t)  =0 \\
  P_+( \rho^+_0 + \sigma_+(x_1,x_2,0,t)) = P_-(\rho^-_0 + \sigma_-(S_-(x_1,x_2,t),t))
 \end{cases}
\end{equation}
where the slip map  \eqref{slip_map_def} is rewritten as 
\begin{equation}
 S_- = (Id_{\Rn{2}} - \zeta_-)\circ (Id_{\Rn{2}} + \tilde{\eta}_+) = Id_{\Rn{2}} + \tilde{\eta}_+ - \zeta_-\circ(Id_{\Rn{2}} + \tilde{\eta}_+).
\end{equation}
Finally, we require the boundary condition
\begin{equation}\label{pert_bc}
 v_-(x_1,x_2,-m,t) \cdot e_3 = v_+(x_1,x_2,\ell,t) \cdot e_3 =0.
\end{equation}
We collectively refer to the evolution, jump, and boundary equations \eqref{perturb}--\eqref{pert_bc} as ``the perturbed problem.''  To shorten notation, for $k\ge 0$ we define
\begin{equation}
 \norm{(\tilde{\eta}, v, \sigma)(t)}_{H^k} = \norm{\tilde{\eta}(t)}_{H^k} + \norm{v(t)}_{H^k} + \norm{\sigma(t)}_{H^k}. 
\end{equation}

\begin{definition}\label{EE_def}
We say that the perturbed problem has property $EE(k)$ for some $k\ge 3$ if there exist $\delta,t_0,C>0$ and a function $F:[0,\delta) \rightarrow \Rn{+}$ satisfying $F(z) \le C z$ for $z\in [0,\delta)$ so that the following hold.  For any $\tilde{\eta}_0, v_0, \sigma_0$ satisfying
\begin{equation}
 \norm{(\tilde{\eta}_0, v_0, \sigma_0)}_{H^k} < \delta
\end{equation}
there exist $(\tilde{\eta},v,\sigma)\in L^\infty((0,t_0);H^3(\Omega))$ so that
\begin{enumerate}
 \item $(\tilde{\eta},v,\sigma)(0) = (\tilde{\eta}_0, v_0, \sigma_0)$,
 \item $\eta(t) = Id + \tilde{\eta}(t)$ is invertible and $\eta^{-1}(t) = Id - \zeta(t)$ for $0\le t<t_0$,
 \item $\tilde{\eta},v,\sigma,\zeta$ solve the perturbed problem on $\Omega \times (0,t_0)$, and
 \item  we have the estimate
\begin{equation}\label{EE_est}
 \sup_{0\le t < t_0} \norm{(\tilde{\eta}, v, \sigma)(t)}_{H^3} \le F(\norm{(\tilde{\eta}_0, v_0, \sigma_0)}_{H^k}).
\end{equation}
\end{enumerate}
\end{definition}

Here the $EE$ stands for existence and estimates, i.e. local-in-time existence of solutions for small initial data, coupled to $L^\infty((0,t_0);H^3(\Omega))$ estimates in terms of the $H^k(\Omega)$ norm of the initial data.  If we were to add the additional condition that such solutions be unique, then this trio could be considered a well-posedness theory for the perturbed problem.

We can now show that property $EE(k)$ cannot hold for any $k\ge 3$.  The proof utilizes the Lipschitz structure of $F$ to show that property $EE(k)$ would give rise to certain estimates of solutions to the linearized equations \eqref{linearized} that cannot hold in general because of Theorem \ref{linear_ill_posed}.

\begin{thm}\label{no_wpk}
 The perturbed problem does not have property $EE(k)$ for any $k\ge 3$.
\end{thm}
\begin{proof}
Suppose by way of contradiction that the perturbed problem has property $EE(k)$ for some $k\ge 3$.  Let $\delta,t_0,C>0$ and $F:[0,\delta)\rightarrow \Rn{+}$ be the constants and function provided by property $EE(k)$.  Fix $n \in \mathbb{N}$ so that $n>C$.  Applying Theorem \ref{linear_ill_posed} with this $n$, $T_0 = t_0/2$, $k\ge 3$, and $\alpha = 1$, we can find $\bar{\eta}, \bar{v}, \bar{\sigma}$ solving \eqref{linearized} so that 
\begin{equation}
 \norm{(\bar{\eta}, \bar{v}, \bar{\sigma})(0)}_{H^k} < \frac{1}{n}
\end{equation}
but
\begin{equation}\label{n_w_20}
 \norm{\bar{v}(t)}_{H^3}\ge \norm{\bar{\eta}(t)}_{H^3} \ge 1 \text{ for }t\ge t_0/2.
\end{equation}
For $\ep>0$ we then define $\bar{\eta}^\ep_0 = \ep  \bar{\eta}(0)$, $\bar{v}^\ep_0 = \ep \bar{v}(0)$, and $\bar{\sigma}^\ep_0 =\ep \bar{\sigma}(0)$. 

Then for $\ep< \delta n$ we have $\norm{(\bar{\eta}^\ep_0, \bar{v}^\ep_0, \bar{\sigma}^\ep_0)}_{H^k}<\delta$, so according to $EE(k)$ there exist  $(\tilde{\eta}^\ep,v^\ep,\sigma^\ep)\in L^\infty((0,t_0);H^3(\Omega))$ that solve the perturbed problem with $(\bar{\eta}^\ep_0, \bar{v}^\ep_0, \bar{\sigma}^\ep_0)$ as initial data and that satisfy the inequality 
\begin{equation}\label{n_w_1}
 \sup_{0\le t < t_0} \norm{(\tilde{\eta}^\ep,v^\ep,\sigma^\ep)(t)}_{H^3} \le F(\norm{(\bar{\eta}^\ep_0, \bar{v}^\ep_0, \bar{\sigma}^\ep_0)}_{H^k}) \le C \ep \norm{(\bar{\eta}, \bar{v}, \bar{\sigma})(0)}_{H^k} < \ep.
\end{equation}
Now define the rescaled functions 
$\bar{\eta}^\ep = \tilde{\eta}^\ep/\ep,  \bar{v}^\ep = v^\ep/\ep, \bar{\sigma}^\ep =  \sigma^\ep/\ep;$ rescaling \eqref{n_w_1} then shows that 
\begin{equation}\label{n_w_2}
 \sup_{0\le t < t_0} \norm{(\bar{\eta}^\ep,\bar{v}^\ep,\bar{\sigma}^\ep)(t)}_{H^3}<1.
\end{equation}
Note that by construction $(\bar{\eta}^\ep,\bar{v}^\ep,\bar{\sigma}^\ep)(0)=(\bar{\eta},\bar{v},\bar{\sigma})(0).$  Our goal is to show that the rescaled functions converge as $\ep \rightarrow 0$ to the solutions $(\bar{\eta},\bar{v},\bar{\sigma})$ of the linearized equations \eqref{linearized}.

We may further assume that $\ep$ is sufficiently small so that 
\begin{equation}\label{n_w_3}
 \sup_{0\le t <t_0} \pnorm{ \ep \bar{\sigma}^\ep(t)}{\infty}   < \hal \inf_{-m\le x_3 \le \ell} \rho_0(x_3)
\end{equation}
and $\ep < 1/(2K_1)$, where $K_1>0$ is the best constant in the inequality 
\begin{equation}
\norm{FG}_{H^2} \le K_1 \norm{F}_{H^2} \norm{G}_{H^2}
\end{equation}
for $3\times 3$ matrix-valued functions $F,G$.  The former condition implies that $\rho_0 + \ep \bar{\sigma}^\ep$ is bounded above and below by positive quantities, whereas the latter guarantees that 
\begin{equation}
 \bar{B}^\ep := (I - (I+\ep D\bar{\eta}^T)^{-1})/\ep
\end{equation} 
is well-defined and uniformly bounded in $L^\infty((0,t_0);H^2(\Omega))$ since
\begin{multline}
 \norm{\bar{B}^\ep}_{H^2} = \norm{ \sum_{n=1}^\infty (-\ep)^{n-1} (D\bar{\eta}^\ep)^n}_{H^2} \le \sum_{n=1}^\infty \ep^{n-1} \norm{(D\bar{\eta}^\ep)^n}_{H^2} \\
\le \sum_{n=1}^\infty (\ep K_1)^{n-1} \norm{D \bar{\eta}^\ep}^n_{H^2}
\le \sum_{n=1}^\infty \frac{1}{2^{n-1}} \norm{\bar{\eta}^\ep}^n_{H^3} < \sum_{n=1}^\infty \frac{1}{2^{n-1}}=2.
\end{multline}
Since $Id + \ep \bar{\eta}^\ep$ is invertible, we may define $\bar{\zeta}^\ep$ via $(Id + \ep \bar{\eta}^\ep)^{-1} = Id - \ep \bar{\zeta}^\ep$, which implies that  $ \bar{\zeta}^\ep  = \bar{\eta}^\ep\circ (Id - \ep \bar{\zeta}^\ep)$.  The  slip map $S_-^\ep: \Rn{2} \times \Rn{+} \rightarrow \Rn{2} \times \{0\}$ is then given by
\begin{equation}
 S^\ep_- = Id_{\Rn{2}} + \ep \bar{\eta}_+^\ep - \ep \bar{\zeta}^\ep_- \circ (Id_{\Rn{2}} + \ep \bar{\eta}^\ep_+).
\end{equation}
The bounds on $\bar{\eta}^\ep$ and the equation satisfied by $\bar{\zeta}^\ep$ then imply that
\begin{equation}\label{n_w_4}
 \sup_{0\le t <t_0} \norm{S_-^\ep(t) - Id_{\Rn{2}}}_{L^\infty} \le 2\ep \sup_{0\le t <t_0} \norm{\bar{\eta}^\ep(t)}_{L^\infty} \le 2 \ep K_2 \sup_{0\le t <t_0} \norm{\bar{\eta}^\ep(t)}_{H^3} < 2\ep K_2, 
\end{equation}
where $K_2>0$ is the embedding constant for the trace map $H^3(\Omega) \hookrightarrow L^\infty(\Rn{2} \times \{0\})$.  This bound allows us to define the normalized slip map $\bar{S}^\ep_- := (S_-^\ep - Id_{\Rn{2}})/\ep$ as a well-defined and uniformly bounded function in $L^\infty((0,t_0);L^\infty(\Rn{2} \times \{0\})$.  Finally, we define the normalized remainder function by
\begin{multline}
 \bar{\mathcal{R}}^\ep(x,t) = \frac{1}{\ep} \int_0^{\ep \bar{\sigma}^\ep(x,t)} (\ep \bar{\sigma}^\ep(x,t) - z)h''(\rho_0(x) + z) dz \\
= \ep \int_0^{ \bar{\sigma}^\ep(x,t)} (\bar{\sigma}^\ep(x,t) - z)h''(\rho_0(x) + \ep z) dz.
\end{multline}
A straightforward calculation employing \eqref{n_w_2} and \eqref{n_w_3} shows that
\begin{equation}\label{n_w_5}
 \sup_{0\le t < t_0} \norm{\bar{\mathcal{R}}^\ep(t)}_{H^3} \le \ep K_3
\end{equation}
for some constant $K_3>0$.

We can now parlay the bounds on $\bar{\eta}^\ep, \bar{v}^\ep, \bar{\sigma}^\ep, \bar{B}^\ep,$ and $\bar{\mathcal{R}}^\ep$ into corresponding bounds on $\dt \bar{\eta}^\ep, \dt \bar{v}^\ep,$ and $\dt \bar{\sigma}^\ep$ and some convergence results.  The first equation in \eqref{perturb} implies that 
$\dt \bar{\eta}^\ep = \bar{v}^\ep$, so that
\begin{equation}\label{n_w_16}
 \sup_{0\le t < t_0} \norm{ \dt \bar{\eta}^\ep(t)}_{H^3} = \sup_{0\le t < t_0}  \norm{  \bar{v}^\ep(t)}_{H^3} \le 1.
\end{equation}
Expanding the second equation in \eqref{perturb}, we find that
\begin{equation}
 (\dt \bar{\sigma}^\ep + \rho_0 \diverge{\bar{v}^\ep}) + \ep (\bar{\sigma}^\ep \diverge{\bar{v}^\ep} - \rho_0 \text{tr}(\bar{B}^\ep D\bar{v}^\ep)  ) - \ep^2 (\bar{\sigma}^\ep \text{tr}(\bar{B}^\ep D\bar{v}^\ep))=0.
\end{equation}
This implies that 
\begin{equation}\label{n_w_11}
\lim_{\ep \rightarrow 0} \sup_{0\le t < t_0} \norm{ \dt \bar{\sigma}^\ep(t) + \rho_0 \diverge{\bar{v}^\ep}(t)  }_{H^2} =0
\end{equation}
and that
\begin{equation}\label{n_w_17}
\sup_{0\le t < t_0} \norm{ \dt \bar{\sigma}^\ep(t)}_{H^2} < K_4 
\end{equation}
for a constant $K_4>0$.  Performing a similar expansion on the third equation in \eqref{perturb} yields
\begin{equation}
 (\dt \bar{v}^\ep + \nab( h'(\rho_0) \bar{\sigma}^\ep + g e_3 \cdot \bar{\eta}^\ep )  ) 
+ \ep (\nab \bar{\mathcal{R}}^\ep - \bar{B}^\ep \nab( h'(\rho_0)\bar{\sigma}^\ep + g e_3 \cdot \bar{\eta}^\ep   )   ) 
+ \ep^2 (\bar{B}^\ep \nab \bar{\mathcal{R}}^\ep)
=0.
\end{equation}
Hence
\begin{equation}\label{n_w_12}
\lim_{\ep \rightarrow 0} \sup_{0\le t < t_0} \norm{ \dt \bar{v}^\ep(t) + \nab( h'(\rho_0) \bar{\sigma}^\ep(t) + g e_3 \cdot \bar{\eta}^\ep(t) ) }_{H^2} =0
\end{equation}
and 
\begin{equation}\label{n_w_18}
\sup_{0\le t < t_0} \norm{ \dt \bar{v}^\ep(t)}_{H^2} < K_5 
\end{equation}
for a constant $K_5>0$.  

We now turn to some convergence results for the jump conditions.  We begin with the second equation in \eqref{pert_jc}, which we expand using the mean-value theorem to get
\begin{multline}\label{n_w_6}
 P_+(\rho^+_0) + P'_+(\alpha^\ep_+ \rho^+_0 + (1-\alpha^\ep_+) \ep \bar{\sigma}_+   ) \ep \bar{\sigma}_+^\ep\\
 = P_-(\rho^-_0) + P'_-(\alpha^\ep_- \rho^-_0 + (1-\alpha^\ep_-) \ep \bar{\sigma}_- \circ(Id_{\Rn{2}} + \ep \bar{S}^\ep_-) ) \ep \bar{\sigma}_-^\ep\circ(Id_{\Rn{2}} + \ep \bar{S}^\ep_-)
\end{multline}
for functions $\alpha^\ep_\pm : \Rn{2}\times\Rn{+} \rightarrow [0,1]$.  From the above bounds on $\bar{\sigma}^\ep$ and $\bar{S}^\ep$ we  have that
\begin{equation}\label{n_w_7}
\sup_{0\le t < t_0} \pnorm{P'_+(\alpha^\ep_+ \rho^+_0 + (1-\alpha^\ep_+) \ep \bar{\sigma}_+   ) - P'_+(\rho^+_0)}{\infty}  \rightarrow 0,
\end{equation}
\begin{equation}
\sup_{0\le t < t_0} \pnorm{P'_-(\alpha^\ep_- \rho^-_0 + (1-\alpha^\ep_-) \ep \bar{\sigma}_- \circ(Id_{\Rn{2}} + \ep \bar{S}^\ep_-) ) - P'_-(\rho^-_0)}{\infty} \rightarrow 0,
\end{equation}
and 
\begin{equation}\label{n_w_8}
\sup_{0\le t < t_0} \pnorm{\bar{\sigma}_-^\ep\circ(Id_{\Rn{2}} + \ep \bar{S}^\ep_-) - \bar{\sigma}_-^\ep }{\infty} \le \sup_{0\le t < t_0} \pnorm{\nab \bar{\sigma}^\ep(t)}{\infty}  \sup_{0\le t < t_0} \pnorm{\ep \bar{S}^\ep(t)}{\infty} \rightarrow 0 \text{ as }\ep \rightarrow 0.
\end{equation}
Since $P_+(\rho^+_0) = P_-(\rho^-_0)$, we may eliminate these terms from equation \eqref{n_w_6} and divide both sides by $\ep$; then employing \eqref{n_w_7}--\eqref{n_w_8}, we deduce that
\begin{equation}\label{n_w_9}
 \sup_{0\le t < t_0} \pnorm{P'_+(\rho^+_0) \bar{\sigma}^\ep_+(t) - P'_-(\rho^-_0) \bar{\sigma}^\ep_-(t)}{\infty} \rightarrow 0 \text{ as }\ep \rightarrow 0.
\end{equation}

For the second equation in \eqref{pert_jc} we first write the normal at the interface as $n^\ep = N^\ep / \abs{N^\ep}$ with
\begin{multline}
 N^\ep = (e_1 + \ep \partial_{x_1} \bar{\eta}^\ep_+) \times (e_2 + \ep \partial_{x_2} \bar{\eta}^\ep_+) \\
= e_3 + \ep (e_1 \times \partial_{x_2} \bar{\eta}^\ep_+ + \partial_{x_1} \bar{\eta}^\ep_+ \times e_2    ) + \ep^2 (\partial_{x_1} \bar{\eta}^\ep_+ \times \partial_{x_2} \bar{\eta}^\ep_+ )
:= e_3 + \ep \bar{N}^\ep.
\end{multline}
As $\ep \rightarrow 0$ we have $\abs{N^\ep} >0$, so we may rewrite the second equation in \eqref{pert_jc} as
\begin{equation}
 (\bar{v}^\ep_+ - \bar{v}^\ep_- \circ (Id_{\Rn{2}} + \ep \bar{S}^\ep_-)   ) \cdot (e_3 + \ep \bar{N}^\ep)=0.
\end{equation}
Clearly $\sup_{0\le t < t_0} \pnorm{\bar{N}^\ep(t)}{\infty}$ is bounded uniformly and
\begin{equation}
\sup_{0\le t < t_0} \pnorm{\bar{v}_-^\ep\circ(Id_{\Rn{2}} + \ep \bar{S}^\ep_-) - \bar{v}_-^\ep }{\infty} \le \sup_{0\le t < t_0} \pnorm{D \bar{v}^\ep(t)}{\infty}  \sup_{0\le t < t_0} \pnorm{\ep \bar{S}^\ep(t)}{\infty} \rightarrow 0 \text{ as }\ep \rightarrow 0.
\end{equation}
Using these, we find that 
\begin{equation}\label{n_w_10}
 \sup_{0\le t < t_0} \pnorm{e_3\cdot(\bar{v}^\ep_+(t) -\bar{v}^\ep_-(t))  }{\infty} \rightarrow 0 \text{ as }\ep \rightarrow 0.
\end{equation}

According to the bound \eqref{n_w_2} and sequential weak-$*$ compactness, we have that up to the extraction of a subsequence (which we still denote using only $\ep$)
\begin{equation}
 (\bar{\eta}^\ep,\bar{v}^\ep,\bar{\sigma}^\ep) \wstar (\bar{\eta}^0,\bar{v}^0,\bar{\sigma}^0) \text{ weakly-}*\text{ in } L^\infty((0,t_0);H^3(\Omega)).
\end{equation}
By lower semicontinuity we know that
\begin{equation}\label{n_w_19}
 \sup_{0\le t < t_0} \norm{(\bar{\eta}^0,\bar{v}^0,\bar{\sigma}^0)(t)}_{H^3} \le 1.
 \end{equation}
On the other hand, by \eqref{n_w_16}, \eqref{n_w_17}, and \eqref{n_w_18}, we know that 
\begin{equation}
\limsup_{\ep \rightarrow 0} \sup_{0\le t < t_0}  \norm{(\dt \bar{\eta}^\ep, \dt \bar{v}^\ep, \dt \bar{\sigma}^\ep)(t)}_{H^2} < \infty.
\end{equation}
By a result in \cite{simon}, we then have that the set $\{(\bar{\eta}^\ep,\bar{v}^\ep,\bar{\sigma}^\ep)\}$ is strongly pre-compact in the space $L^\infty((0,t_0);H^{11/4}(\Omega))$, so 
\begin{equation}
 (\bar{\eta}^\ep,\bar{v}^\ep,\bar{\sigma}^\ep) \rightarrow (\bar{\eta}^0,\bar{v}^0,\bar{\sigma}^0) \text{ strongly in } L^\infty((0,t_0);H^{11/4}(\Omega)).
\end{equation}
This strong convergence, together with the convergence results \eqref{n_w_11}, \eqref{n_w_12} and the equation $\dt \bar{\eta}^\ep=\bar{v}^\ep$, implies that
\begin{equation}
 (\dt \bar{\eta}^\ep, \dt \bar{v}^\ep, \dt \bar{\sigma}^\ep) \rightarrow (\dt \bar{\eta}^0, \dt \bar{v}^0, \dt \bar{\sigma}^0) \text{ strongly in } L^\infty((0,t_0);H^{7/4}(\Omega)),
\end{equation}
and that
\begin{equation}\label{n_w_13}
\begin{cases}
 \dt \bar{\eta}^0 = \bar{v}^0 \\
 \dt \bar{\sigma}^0 + \rho_0 \diverge{\bar{v}^0} = 0  \\
 \dt \bar{v}^0 + \nab( h'(\rho_0) \bar{\sigma}^0 + g e_3 \cdot \bar{\eta}^0 )=0.
\end{cases} 
\end{equation}
We may pass to the limit in the initial conditions $(\bar{\eta}^\ep,\bar{v}^\ep,\bar{\sigma}^\ep)(0)=(\bar{\eta},\bar{v},\bar{\sigma})(0)$ to find that 
\begin{equation}\label{n_w_15}
(\bar{\eta}^0,\bar{v}^0,\bar{\sigma}^0)(0)=(\bar{\eta},\bar{v},\bar{\sigma})(0)
\end{equation}
as well.  We now derive the jump and boundary conditions for the limiting functions. The index $11/4$ is sufficiently large to give $L^\infty((0,t_0);L^\infty)$ convergence of  $(\bar{\eta}^\ep,\bar{v}^\ep,\bar{\sigma}^\ep)$ when restricted to $\{x_3=0\}$, $\{x_3=-m\}$, and $\{x_3=\ell\}$, i.e. the interface and the lower and upper boundaries.  Combining this with \eqref{n_w_9} and \eqref{n_w_10}, we deduce that
\begin{equation}
 P'_+(\rho^+_0) \bar{\sigma}^0_+ = P'_-(\rho^-_0) \bar{\sigma}^0_- \text{ on } \{x_3=0\},
\end{equation}
\begin{equation}
 (\bar{v}^0_+ -\bar{v}^0_-)\cdot e_3 = 0  \text{ on } \{x_3=0\},
\end{equation}
and that
\begin{equation}\label{n_w_14}
 \bar{v}_+^0 \cdot e_3 =0 \text{ on } \{x_3=\ell\} \text{ and } \bar{v}_-^0 \cdot e_3 =0 \text{ on } \{x_3=-m\}.
\end{equation}

Now, according to \eqref{n_w_13}--\eqref{n_w_14}, $(\bar{\eta}^0,\bar{v}^0,\bar{\sigma}^0)$ are solutions to \eqref{linearized} and the corresponding jump and boundary conditions on $\Omega \times (0,t_0)$ that satisfy the initial condition \eqref{n_w_15}.  Then according to Theorem \ref{linear_uniqueness}, 
\begin{equation}
 (\bar{\eta}^0,\bar{v}^0,\bar{\sigma}^0) = (\bar{\eta},\bar{v},\bar{\sigma}) \text{ on } \Omega \times [0,t_0).
\end{equation}
Hence we may chain together inequalities \eqref{n_w_19} and \eqref{n_w_20} to get
\begin{equation}
 2 < \sup_{t_0/2 \le t < t_0} \norm{(\bar{\eta}^0,\bar{v}^0,\bar{\sigma}^0)(t)}_{H^3} \le    \sup_{0\le t < t_0} \norm{(\bar{\eta}^0,\bar{v}^0,\bar{\sigma}^0)(t)}_{H^3} \le 1, 
\end{equation}
which is a contradiction.  Therefore, the perturbed problem does not have property $EE(k)$ for any $k\ge 3$.
\end{proof}

\pagebreak


\end{document}